\documentclass[11pt,oneside]{amsart}
\usepackage{graphicx, amsmath, amssymb, amscd, amsthm, euscript, psfrag, amsfonts,bm,color}

\DeclareMathAlphabet{\mathpzc}{OT1}{pzc}{m}{it}

\newtheorem{thm}{Theorem}[section]

\newtheorem{Rmk}[thm]{Remark}

\newtheorem{prop}[thm]{Proposition}

\newtheorem{cor}[thm]{Corollary}
\newtheorem{lem}[thm]{Lemma}
\newtheorem{lemma}[thm]{Lemma}

\newtheorem{Def}[thm]{Definition}

\numberwithin{equation}{section}

\newcommand{\be}{begin{equation}}

\newcommand{\bH}{\mathbb H}

\newcommand{\q}{\mathbb{Q}}

\newcommand{\e}{{\varepsilon}}
\newcommand{\z}{\mathbb{Z}}
\renewcommand{\q}{\mathbb{Q}}

\renewcommand{\c}{\mathbb{C}}
\newcommand{\br}{\mathbb{R}}

\newcommand{\ba}{\backslash}

\newcommand{\G}{\Gamma}

\newcommand{\SL}{\operatorname{SL}}\newcommand{\PSL}{\operatorname{PSL}}
\newcommand{\GL}{\operatorname{GL}}
\newcommand{\bp}{\begin{pmatrix}}
\newcommand{\ep}{\end{pmatrix}}

\renewcommand{\bp}{{\rm bp}}

\newcommand{\SO}{\operatorname{SO}}

\newcommand{\hpz}{\mathpzc{h}}

\newcommand{\Hor}{\mathpzc{H}}

\newcommand{\op}{\operatorname}

\newcommand{\vs}{\vskip 5pt}

\newcommand{\Isom}{\operatorname{Isom}}

\renewcommand{\be}{\begin{equation}}
\newcommand{\ee}{\end{equation}}
\newcommand{\ben}{\begin{enumerate}}
\newcommand{\een}{\end{enumerate}}

\newcommand{\F}{\operatorname{F}\!}
\newcommand{\RF}{\operatorname{RF}\!}

\newcommand{\FM}{\F M}
\newcommand{\RK}{\operatorname{RF}_kM}

\newcommand{\RFM}{\op{RF}M}

\newcommand{\core}{\text{core}}

\newcommand{\FS}{F^{*}}
\newcommand{\FL}{F_{\Lambda}}\newcommand{\CS}{{\mathcal C}^{*}}
\newcommand{\CL}{{\mathcal C}_{\Lambda}}
\newcommand{\hull}{\op{hull}}
\newcommand{\oxH}{\overline{xH}}

\newcommand{\bq}{\begin{quotation}}
\newcommand{\eq}{\end{quotation}}

\newcommand{\m}[1]{\mathbb{ #1}}
\newcommand{\mc}[1]{\mathcal{ #1}}

\newcommand{\oM}{\overline{M}}
\begin{document}

\title[]{Geodesic Planes in  geometrically finite acylindrical $3$-manifolds}

\author{Yves Benoist}
\address{CNRS, Universite Paris-Sud, IMO, 91405 Orsay, France}
\email{yves.benoist@u-psud.fr}

\author{Hee Oh}
\address{Mathematics department, Yale university, New Haven, CT 06520 
and Korea Institute for Advanced Study, Seoul, Korea}

\email{hee.oh@yale.edu}
\thanks{Oh was supported in part by NSF Grant.}

\begin{abstract} {Let $M$ be a geometrically finite acylindrical 
hyperbolic $3$-manifold and let $M^*$
denote the interior of the convex core of $M$. 
We show that any geodesic plane in $M^*$ is either closed or dense, and  
that there are only countably many closed geodesic planes in $M^*$.
These results were obtained in \cite{MMO} and \cite{MMO2} when $M$ is convex cocompact. 

As a corollary, we obtain that when $M$ covers an arithmetic hyperbolic $3$-manifold $M_0$, the topological behavior 
of a geodesic plane in $M^*$ is governed by that of the
corresponding plane in $M_0$. We construct a counterexample of this phenomenon when $M_0$ is non-arithmetic. }\end{abstract}
\maketitle
\tableofcontents

\section{Introduction}

\subsection{Geometrically finite acylindrical manifolds} Let $M=\Gamma \ba \bH^3$ be a complete, oriented hyperbolic $3$-manifold 
presented as a quotient of hyperbolic space by a discrete
subgroup $$\Gamma \subset G=\op{Isom}^+(\bH^3).$$

We denote by $\Lambda=\Lambda({\Gamma})$ the limit set of $\Gamma$ 
in the sphere at infinity $S^2$ and by $\Omega$ the domain
of discontinuity;  $\Omega=S^2-\Lambda$.
The convex core of $M$ is the smallest  closed convex subset 
containing all periodic geodesics in $M$, or equivalently
$$\text{core}(M):=\Gamma\ba \text{hull}(\Lambda)\subset M$$
is the quotient of the convex hull of the limit set $\Lambda$
by $\Gamma$.
We denote by $M^{*}$ the interior of the convex core of $M$.
Note that $M^*$ is non-empty if and only if $\Gamma$ is Zariski dense in $G$.

A geodesic plane $P$ in $M$ is the image $f(\bH^2)\subset M$ of a totally geodesic immersion $f:\bH^2\to M$
of the hyperbolic plane into $M$. 
By a geodesic plane $P^{*}$ in  $M^{*}$, we mean
the non-trivial intersection 
$$
P^{*}=P\cap M^{*} \ne \emptyset
$$
of a geodesic plane in $M$ with the interior of the convex core.
Note that a  plane $P^*$ in $M^*$ is always connected and that
 $P^*$ is closed in $M^*$ if and only if $P^*$ is properly immersed in $M^*$ \cite[\S 2]{MMO2}.

We say $M$ is geometrically finite 
if the unit neighborhood of $\core(M)$ has finite volume. When $\core(M)$ is compact, $M$ is called convex cocompact.

The closures of geodesic planes in $M^{*}$ have been classified for convex cocompact acylindrical manifolds
by McMullen, Mohammadi and the second named author in \cite{MMO} and \cite{MMO2}. The main aim of this paper is to extend those classification results 
to geometrically finite acylindrical manifolds.
In concrete terms, the new feature of this paper compared to 
\cite{MMO} and \cite{MMO2} is that we allow the existence of cusps in $M^*$.

For a geometrically finite manifold $M$, 
the condition that $M$ is acylindrical is a topological one, it means  that its compact core $N$ (called the Scott core) is acylindrical, i.e.,
$N$ has incompressible boundary, and every essential cylinder in $N$ is boundary parallel \cite{Th}.

In the case when the boundary of the convex core of $M$ is totally geodesic,  we call
 $M$ {\it rigid acylindrical}. The class of rigid acylindrical hyperbolic $3$-manifolds $M$ 
includes those for which $M^*$ is obtained by ``cutting'' a finite volume complete hyperbolic $3$-manifold
$M_0$   along a properly embedded compact geodesic surface $S\subset M_0$ 
(see \S \ref{secarithm} for explicit examples).   
We remark that a geometrically finite acylindrical manifold is quasiconformal conjugate to a unique
geometrically finite  rigid acylindrical manifold (\cite{Th1}, \cite{Mc}). 

\subsection{Closures of geodesic planes} 
Our main theorem is the following:
\begin{thm} 
\label{m1} 
Let $M$ be a geometrically finite  acylindrical  hyperbolic $3$-manifold.
Then any geodesic plane $P^{*}$  is either closed or dense in $M^{*}$. Moreover, there are only countably many
closed geodesic planes $P^*$ in $M^{*}$.
\end{thm}

In the rigid case, each geodesic boundary component has a fundamental group which is a cocompact lattice in
$\PSL_2(\br)$, up to conjugation. This rigid structure forces any closed plane $P^{*}$ in $M^{*}$ to be a part of a closed plane $P$ in $M$:

\begin{thm} 
\label{m2} 
If $M$ is rigid in addition, 
then any geodesic plane $P$ in $M$ intersecting $M^{*}$ non-trivially is either closed or dense in $M$. 
\end{thm}

We do not know whether Theorem \ref{m2} can be extended to a non-rigid $M$ or not. This is unknown even when $M$ is convex compact, as remarked in \cite{MMO2}.
 
When $M$ is a cover of an arithmetic hyperbolic $3$-manifold $M_0$, the topological behavior 
of a geodesic plane in $M^*$ is governed by that of the
corresponding plane in $M_0$:

\begin{thm} 
\label{arith1} 
Let $M=\Gamma\ba \bH^3$ be a geometrically finite  acylindrical  manifold which covers an arithmetic 
manifold $M_0=\Gamma_0\ba \bH^3$ of finite volume. Let $p:M\to M_0$ be the covering map.
Let $P\subset M$ be a geodesic plane with $P^*=P\cap M^*\ne \emptyset$.
Then
\begin{enumerate}
\item $P^*$ is closed in $M^*$ if and only if $p(P)$ is closed in $M_0$;
\item $P^*$ is dense in $M^*$ if and only if $p(P)$ is dense in $M_0$.
\end{enumerate}
In the when $M$ is rigid, we can replace $M^*$ by $M$ in the above two statements. \end{thm}

Theorem \ref{arith1} is not true in general without the arithmeticity assumption on $M_0$: 

\begin{thm}\label{nono}
There exists a non-arithmetic closed  hyperbolic $3$-manifold  $M_0$, covered by a geometrically finite rigid acylindrical manifold $M$
such that there exists a properly immersed geodesic plane $P$ in $ M$ with $P\cap M^*\ne \emptyset$ whose image $p(P)$ is dense in $M_0$, where
$p:M\to M_0$ is the covering map.
\end{thm}
\medskip

We also prove the following theorem for a general geometrically finite manifold of infinite volume:

\begin{thm} 
\label{m4}
Let $M$ be a geometrically finite hyperbolic $3$-manifold of infinite volume. 
Then there are only finitely many geodesic planes  in $M$
that are contained in $\text{core}(M)$ and all of them are closed with finite area.
\end{thm}

 Theorem \ref{m4} is generally false for a finite volume manifold. For instance, an arithmetic hyperbolic $3$-manifold
with one closed geodesic plane contains infinitely many closed geodesic planes, which are obtained via Hecke operators.
\subsection{Orbit closures in the space of circles}
We now formulate a stronger version of Theorems \ref{m1} and \ref{m2}. Let  $\mathcal C$ denote the space  of all oriented circles in $S^2$. 
We identify $G:=\op{Isom}^+(\bH^3)$ with $\PSL_2(\m C)$, considered as a simple real Lie group, and $H:=\op{Isom}^+(\bH^2)$ with $\PSL_2( \m R) $, so that
we have a natural isomorphism
 $$\mc C \simeq G/H .$$
The following subsets of $\mathcal C$ play  important roles in our discussion and we will keep the notation
throughout the paper:
\begin{align*} 
\mathcal C_\Lambda 
&=\{C\in \mathcal C: C\cap \Lambda\ne \emptyset\};
\\ 
\mathcal C^*
&=\{C\in \mathcal C: C\text{ separates $\Lambda$} \}.
\end{align*}
The condition {\it $C$ separates $\Lambda$} means that both connected components
of $S^2-C$ meet $\Lambda$.
Note that $\mc C_\Lambda$ is closed  in $\mathcal C$ and that $\mathcal C^*$ is open in $\mathcal C$.
If the limit set $\Lambda$ is connected, 
 then $\mathcal C^*$ is a dense subset  {\it of }  $\mathcal C_\Lambda$  \cite[Corollary 4.5]{MMO}
but $\mathcal C^*$ is not contained in $\mathcal C_\Lambda$ in general.

The classification of closures of geodesic planes follows from the classification of $\Gamma$-orbit closures of circles
(cf. \cite{MMO}). In fact,
the following theorem strengthens both Theorems \ref{m1} and
\ref{m2} in two aspects: it describes possible orbit closures of single $\Gamma$-orbits as well as of any $\Gamma$-invariant subsets of
$\mathcal C^*$.

\begin{thm} 
\label{cir} 
Let $M=\Gamma\ba \bH^3$ be a geometrically finite 
acylindrical  manifold. Then \begin{enumerate}
\item  Any $\Gamma$-invariant subset of $\mathcal C^*$ is either dense or a  union of finitely many closed
$\Gamma$-orbits in $\mathcal C^*$;
\item  There are at most countably many closed $\Gamma$-orbits in $\mc C^*$; and
\item  For $M$  rigid, any $\Gamma$-invariant subset of $\mathcal C^*$ is either dense or a  union of finitely many closed
$\Gamma$-orbits in $\mathcal C_\Lambda$.
\end{enumerate}
\end{thm}

 We also present a reformulation of Theorem \ref{m4}  in terms of circles:
\begin{thm} 
Let $M=\Gamma\ba \m H^3$ be a  geometrically finite manifold of infinite volume. Then there are only
finitely many $\Gamma$-orbits of circles  contained in $\Lambda$. Moreover each of these orbits is closed
in $\mc C_\Lambda$.
\end{thm}
\subsection{Strategy and organization}
Our approach follows the same lines as the approaches of \cite{MMO} and \cite{MMO2}, and we tried to keep the same notation
from those papers as much as possible. 

Below is the list of subgroups of $G=\operatorname{PSL}_2(\c)$ which will be used throughout the paper:
\begin{align*} 
H&=\operatorname{PSL}_2(\br)\\
A&=\left\{ a_t=\begin{pmatrix} e^{t/2} & 0 \\ 0 &
e^{-t/2}\end{pmatrix} : t\in \m R\right\}\\
K&=\op{PSU}(2)\\ 
N& =\left\{u_t=
\begin{pmatrix} 1 & t \\ 0 & 1\end{pmatrix}: 
t\in \mathbb C \right\}\\
U&=\{u_t: t\in \br\}\quad\text{and}\\
V&=\{u_t: t\in i\br \}.
\end{align*}

We can identify $M=\Gamma\ba \bH^3$  with the double coset space $\Gamma\ba G/K$
and  its oriented frame bundle $\FM$  with
the homogeneous space $\Gamma\ba G$.
By the duality between  $\Gamma$-orbits in $\mc C$ and $H$-orbits in $\FM$,  Theorem \ref{cir}
 follows from the classification of $H$-invariant closed subsets in $F^*$ where $F^*$ is the $H$-invariant subset of $\FM$ such that
$\Gamma\ba \mathcal C^*=F^*/H$ (cf. \cite{MMO}, \cite{MMO2}).
Working in the frame bundle $\FM$ enables  us to use 
the geodesic flow,  the horocyclic flow as well as  the ``imaginary horocyclic flow''
given by the right-translation actions of  $A$, $U$, and $V$ in the space $\Gamma\ba G$  respectively.
\vs

The results in \S\S \ref{secprelim}-\ref{sechomdyn} hold for a general 
Zariski dense discrete subgroup  $\Gamma$. 

In \S \ref{secprelim}, we review some basic definitions and notations to be used throughout the paper.
The notation $\RFM$ denotes  the {\it renormalized frame bundle of $M$}, which is
 the closure of the union of all  periodic $A$-orbits in $\FM$ (see also \eqref{fr}). 
 We consider the $H$-invariant subset $F_\Lambda$ of $\FM$ which corresponds to $\mc C_\Lambda$:
 $$  \Gamma \ba \mathcal C_\Lambda= F_\Lambda/H .$$ We set $F_\Lambda^*=F_\Lambda\cap F^*$.
 When $\Lambda$ is connected, $F_\Lambda^*=F^*$, but not in general.

In \S \ref{seclimorb},
we study the closure of an orbit $xH$ in $F_\Lambda$ which accumulates on an orbit $yH$ with non-elementary stabilizer.
We prove that $xH$ is dense in $F_\Lambda$ if $y\in  F^*$ (Proposition \ref{latt}), and also
present a condition for the density of $xH$ in $F_\Lambda$ when $y\not\in F^*$ (Proposition \ref{gbb}).

In \S \ref{secperorb}, we prove that if  the closure $\overline{xH}$ 
 contains a periodic $U$-orbit in $ F_\Lambda^*$, then $xH$ is
 either locally closed or dense in $F_\Lambda$ (Proposition \ref{cd2}).

In \S \ref{secrethor}, we  
recall a closed $A$-invariant subset  $\RK\subset \RFM$, $k>1$, with {\it $k$-thick} recurrence properties for the horocycle flow,
which was introduced in \cite{MMO2}. This set may be empty in general.
We show that the thick recurrence property remains preserved even after we remove a neighborhood of finitely many cusps
from $\RK$. That is,
setting 
\be\label{wk} W_{k, R}:=\RK -\Hor_R\ee
where $\Hor_R$ is a  neighborhood of finitely many cusps in $\FM$ of  depth $R\gg 1$,
we show that for any $x\in W_{k,R}$,
the set $T_x:=\{t\in \br: xu_t\in W_{k,R}\}$  is {\it $4k$-thick at $\infty$}, in the sense
that for all sufficiently large $r>1$ (depending on $x$),
$$T_x\cap \left( [-4kr, -4r]\cup [r, 4kr]\right) \ne \emptyset$$
(Proposition \ref{thick}).

In \S \ref{sechomdyn}, we present a technical proposition (Proposition \ref{ffin})
which ensures the density of $xH$ in $F_\Lambda$ when $xH$ intersects a compact subset $W\subset F_\Lambda^*$
with the property that the return times $\{t: xu_t\in W\}$ is
$k$-thick at $\infty$ for every $x\in W$.
The proof of this proposition uses the notion of $U$-minimal sets with respect to $W$, together with the polynomial divergence property
 of unipotent flows of two nearby points,  which goes back to Margulis' proof of Oppenheim conjecture \cite{Mar}. 
 In the context when the return times of the horocyclic flow
is only $k$-thick (at $\infty$), this argument was  used  in \cite{MMO} and \cite{MMO2}  in order to construct a piece of $V$-orbit inside the  closure  of $xH$, which can then be pushed to the density of $xH$ in $F_\Lambda$ using Corollary \ref{vp}.
\vs

In \S\S \ref{seccloden}-\ref{secuniper}, we make an additional assumption that
 $\Gamma$ is geometrically finite, which implies that $\RFM$ is a union of a compact set and finitely many cusps. Therefore 
 the set $W_{k,R}$ in \eqref{wk} is compact when $\Hor_R$ is taken to be a neighborhood of all cusps in $\FM$. 

In \S \ref{seccloden},
we apply the result in \S 6  to an orbit $xH\subset F_\Lambda^*$ intersecting $W_{k, R}$
when $\overline{xH}$ does not contain a periodic $U$-orbit. 
Combined with the results in \S \ref{secperorb}, we conclude that any $H$-orbit intersecting $\RF_k M\cap F^*$
is either locally closed or dense in $F_\Lambda$
(Theorem \ref{rt}).

In \S \ref{secnonele}, we prove that any locally closed orbit $xH$ intersecting $\RF_k M\cap F^*$ has
a non-elementary stabilizer and intersects $\RF_k M\cap F^*$ as a closed subset (Theorem \ref{ns}).

In \S \ref{secuniper}, we give an interpretation of
the results obtained so far in terms of $\Gamma$-orbits of circles for a general geometrically finite Zariski dense subgroup.
(Theorem \ref{circperf}).

In \S \ref{secorbcon}, we prove that any orbit $xH$ included in $\RFM$ is closed 
of finite volume, and that only finitely many such orbits exist, proving Theorem \ref{m4}
(Theorem \ref{hr}).

In \S \ref{secacyman}, we specialize to
the case where $M=\Gamma\ba \m H^3$ is a
geometrically finite acylindrical manifold.
When $M$ is a convex cocompact acylindrical manifold, $\Lambda$
is a Sierpinski curve of positive modulus, that is, there exists $\e>0$ such that the modulus of the annulus between any two components
 $B_1, B_2$ of $\Omega:=S^2-\Lambda$ satisfy
$$\text{mod}(S^2- (\overline B_1\cup \overline B_2))\ge \e$$ 
 \cite[Theorem 3.1]{MMO2}.

When $M$ has cusps, the closures of some components of $\Omega$ meet each other, and hence $\Lambda$ is not even a Sierpinski curve.
 Nevertheless, under the assumption that $M$ is a geometrically finite acylindrical manifold,
 $\Lambda$ is still a {\it quotient of a Sierpinski curve of positive modulus}, in the sense that
  we can present $\Omega$ as the disjoint union $\bigcup_\ell T_\ell$ where $T_\ell$'s are maximal {\it trees of components of $\Omega$}
  so that 
  $$\inf_{\ell\ne k}\;\; \text{mod}(S^2- (\overline T_\ell \cup \overline T_k))>0$$ 
 (Theorem
  \ref{sier}). This analysis enables us to show that every separating circle $C$
intersects $\Lambda$ as a Cantor set of positive modulus (Theorem \ref{cm}).
This implies that
{\it every} $H$-orbit in $F^*$ intersects a compact subset
$W_{k, R}$ for $k$ sufficiently large (Corollary \ref{ffk}). 
We prove Theorem \ref{cir} in terms of $H$-orbits on $\FM$ (Theorem \ref{mainacy}).

In \S \ref{sarith},   we prove Theorem \ref{arith1} 
and give a counterexample when $M_0$ is not arithmetic (Proposition \ref{na}), proving Theorem \ref{nono}.
 We also give various examples of
geometrically finite rigid acylindrical  $3$-manifolds.

\vs

\subsection{Comparison of the proof of Theorem \ref{m1} with the convex cocompact case}
For readers who are familiar with the work \cite{MMO2}, we finish the introduction with a brief account on some of  essential differences
in the dynamical aspects of the proofs between the present work and \cite{MMO2}
for the case when $M$ is a geometrically finite acylindrical manifold.  
The main feature of $M$ being geometrically finite acylindrical is that $\RFM H=\RF_k M H$ for all sufficiently large $k$ (\S \ref{secacyman}).

Consider  an orbit $xH$ in $F^*$ and set $X:=\overline{xH}$.
When $M$ is convex cocompact,  the main strategy of \cite{MMO2} is to analyze $U$-minimal subsets 
of $X$ with respect to $\RF_kM$. That is, unless $xH$ is locally closed, it was shown, using the thickness principle
and the polynomial divergence argument, that any such
$U$-minimal subset is invariant under a connected semigroup $L$ transversal to $U$.
Then pushing further, one could find an $N$-orbit inside $X$, which implies $X=F_\Lambda$.
In the present setting,
if $X\cap F^*$ happens to contain a periodic $U$-orbit, which is a generic case a posteriori, 
then any $U$-minimal subset of $X$ relative to $\RF_kM$ is a periodic $U$-orbit and hence is not invariant
by any other subgroup but $U$. Hence the aforementioned strategy does not work. However it turns out that
this case is simpler to handle, and that's what \S 4 is about. See Proposition \ref{cd2}.

Another important ingredient of \cite{MMO2} is that 
a closed $H$-orbit in $F^*$ has a non-elementary stabilizer; this was used to conclude
$X=F_\Lambda$ whenever $X$ contains a closed $H$-orbit in $F^*$ properly, as well as to establish
the countability of such orbits, and its proof relies
on the absence of parabolic elements. In the present setting, we show that any locally closed
$H$-orbit in $F^*$ has a non-elementary stabilizer, regardless of the existence of parabolic elements of $\Gamma$.
This is Theorem \ref{ns}. 

When $M$ is not convex cocompact,
 $\RFM$ is not compact and neither is $\RF_kM$, which presents an issue in the polynomial divergence argument. In Proposition \ref{thick},
we show that the remaining compact set after removing horoballs
from $\RF_kM$
has the desired thickness property for $U$-orbits and hence can be used as a replacement
of $\RF_kM$.  We remark that  the {\it global} thickness of the return times of $U$-orbits to $\RF_kM$ was crucial in
 establishing this step. The proof of Theorem \ref{m1}  combines these  new ingredients with the techniques
developed in \cite{MMO2}. 

We would also like to point out that we made efforts in trying to write this paper in a greater general setting. For instance, we
  establish  (Theorem \ref{circperf}):
  \begin{thm}
 Let $\Gamma$ be a  geometrically finite Zariski dense subgroup of $G$. Then for any $C\in \mathcal C^*$
 such that $C\cap \Lambda$ contains a uniformly perfect Cantor set,
the orbit $\Gamma (C)$ is either discrete or dense in $\mathcal C_\Lambda$.
\end{thm}

With an exception of \S \ref{secacyman}, we have taken a more homogeneous dynamics  viewpoint overall
in this paper, hoping that this  perspective will also be useful in some context.
\vs

\noindent{\bf Acknowledgement:}
The present paper heavily depends on the previous works \cite{MMO} and \cite{MMO2}
by McMullen, Mohammadi, and the second named author. We are grateful to Curt McMullen
for clarifying the notion of the acylindricality of a geometrically finite manifold. We would like to thank Curt McMullen and Yair Minsky 
for their help in writing \S \ref{secacyman}, especially the the proof of Theorem \ref{sier}.
\section{Preliminaries}
\label{secprelim}
In this section, we set up a few notations which will be used throughout the paper
and review some definitions.
\vs

Recall that $G$ denotes the simple, connected real Lie group $\PSL_2(\mathbb C)=\op{Isom}^+(\bH^3)$.
The action of $G$ on $\bH^3=G/K$ extends continuously to a conformal action of $G$
on the Riemann sphere $S^2=\hat \c \cup \{\infty\}$ 
and the union $\bH^3\cup S^2$ is compact.

The group $G$ can be identified with the oriented frame bundle $\op{F}\bH^3$. 
Denote by $\pi$ the natural projection $ G=\op{F}\bH^3 \to \bH^3$.
If $g\in G$ corresponds to a frame $(e_1, e_2, e_3)\in \op{F} \bH^3 $,
we define $g^{+}, g^{-}\in S^2$ to be the forward and backward end points of the directed geodesic tangent to $e_1$ respectively.
The action of $A$ on $G=\op{F} \bH^3 $ defines the 
frame flow,
and we have
\be
\label{gpm}
g^{\pm} =\lim_{t\to \pm  \infty} \pi(ga_{t})
\ee 
where the limit is taken in the compactification $\bH^3\cup S^2$.

For $g\in G$, the image  $\pi(gH)$ is a geodesic plane in $\bH^3$
and  
$$
(gH)^{+} :=\{(gh)^{+}: h\in H \}\subset  S^2
$$ 
is an oriented circle which bounds the plane $\pi(gH)$.

The correspondences 
$$
gH\to (gH)^{+} 
\;\;\text{and}\;\;
gH\to \pi(gH)
$$
give rise to bijections of $G/H$ with the space $ \mathcal C$ of all oriented circles in $S^2$, as well as with the space of all oriented
geodesic planes of $\bH^3$.

A horosphere (resp. horoball) in $\bH^3$ is a Euclidean sphere (resp. open Euclidean ball) tangent to $S^2$
and a horocycle in $\m H^3$ is a Euclidean circle tangent at $S^2$.
A horosphere is the image an  $N$-orbit under $\pi$ while a horocycle is the image of a  $U$-orbit under $\pi$. 

\subsection{Renormalized frame bundle} 
Let $\Gamma$ be a non-elementary discrete subgroup of $G$ and  $M=\Gamma \ba \bH^3$. We denote by $\Lambda=\Lambda(\Gamma)
\subset S^2$ the limit set of $\Gamma$.
We can identify the oriented frame bundle $\FM$ with $\Gamma \ba G$. With abuse of notation,
we also denote by $\pi$ for the canonical projection $\FM \to M$.
Note that, for $x=[g]\in \FM$, the condition $g^{\pm}\in \Lambda$ 
does not depend on the choice of a representative.
The renormalized frame bundle $\RFM\subset \FM $ is defined as
\be
\label{fr}
\RFM=\{[g] \in \Gamma\ba G: g^{\pm}\in \Lambda\} ,
\ee
in other words, it is the closed subset of $\FM$ consisting of all frames $(e_1, e_2, e_3)$ such that
$e_1$ is tangent to a complete geodesic contained in the core of $M$.

We define the following $H$-invariant subsets of $\FM$:
\begin{align*} F_\Lambda&=\{[g]\in \Gamma\ba G: (gH)^+ \cap \Lambda \ne \emptyset\}=(\RFM ) N H ,\\
F^{*}& =\{[g]\in \G\ba G : (gH)^+ \text{separates $\Lambda$}\},\\
F^*_\Lambda &=F_\Lambda\cap F^*.
\end{align*}
Note that $F^*=\{x\in \FM: \pi(xH)\cap M^*\ne \emptyset\}$ is an open subset of $\FM$ 
and that $F^*\subset F_\Lambda$ when $\Lambda$ is connected. In particular,
when $\Lambda$ is connected, $F_\Lambda$ has non-empty interior.

We recall that a circle $C\in \mc C$ is {\it separating}
or {\it separates $\Lambda$} 
if both connected components
of $S^2 -C$ intersect $\Lambda$.
By analogy, a frame $x\in F^*$ will be called a {\it separating frame}.    

In the identification $\mathcal C=G/H$,  the sets  $F_\Lambda$, $F^{*}$ and $F^*_\Lambda$
satisfy
$$\Gamma\ba \mathcal C_\Lambda = F_\Lambda /H,\;\;
\Gamma\ba \mathcal C^{*} = F^{*} /H  \;\; \text{ and } \;\; \Gamma\ba \mathcal C^*_\Lambda = F^*_\Lambda /H ,$$
where 
$$
\mathcal C^*_\Lambda:=\mathcal C_\Lambda\cap \mathcal C^* .$$

Since $\Gamma$ acts properly on the domain of discontinuity $\Omega$, 
$\Gamma C$ is closed in $\mc C$ for any circle $C$ with $C\cap \Lambda=\emptyset$. 
For this reason, we only consider $\Gamma$-orbits  in $\mc C_\Lambda$
or, equivalently, $H$-orbits in $F_\Lambda$.

\subsection{Geometrically finite groups} We give a characterization of geometrically finite groups in terms of their limit sets
(cf. \cite{Bow}, \cite{Ma1}, \cite{MT}).
A limit point $\xi\in \Lambda$ is called {\it radial} if any geodesic ray toward $\xi$ has an accumulation point in $M$
and  {\it parabolic} if $\xi$ is fixed by a
parabolic element of $\Gamma$. In the group $G=\PSL_2(\c)$,
parabolic elements are precisely unipotent elements of $G$; any parabolic element in $G$ is conjugate to
the matrix $\begin{pmatrix} 1& 1\\ 0 &1\end{pmatrix}$.
When $\xi$ is parabolic,
its stabilizer $\op{Stab}_{\Gamma}\xi$ in $\Gamma$ is virtually abelian, and its rank is called the rank of $\xi$.
We denote by $\Lambda_r$ the set of radial limit points and by $\Lambda_p$ the set of parabolic fixed points.
The fixed points of hyperbolic elements of $\Gamma$ are contained in $\Lambda_r$ and
form a dense subset of $\Lambda$.
More strongly, the set of pairs of attracting and repelling fixed points of hyperbolic elements
is a dense subset of $\Lambda\times \Lambda$ \cite{Eb}.
\medskip

A Zariski dense discrete subgroup $\Gamma$ of $G$ is called
geometrically finite if it satisfies one of the following
equivalent conditions:
\begin{enumerate}
\item the core of $M$  has finite volume;
\item $\Lambda=\Lambda_r\cup \Lambda_p.$
\end{enumerate}

If $\Lambda=\Lambda_r$, or equivalently, if the core of $M$ is compact, thne $\Gamma$ is called convex cocompact.

We say that a horoball $\mathpzc h$ is  based at $\xi\in S^2$ if it is tangent  to $S^2$ at $\xi$.
If $\xi \in \Lambda_p$, then for any fixed horoball $\hpz \subset \bH^3$ 
based at $\xi\in S^2$,
its $\Gamma$-orbit $\Gamma \hpz$ is closed in the space of horoballs in $\bH^3$.
Given a horoball $\hpz$ and $R>0$, we will write $\hpz_R$ for the horoball contained in $\hpz$ whose distance to the boundary of $\hpz$ is
$R$.

Suppose that $\Gamma$ is  geometrically finite. Then
there exist finitely many horoballs  $\hpz^{1}, \cdots, \hpz^m$ in $\bH^3$
corresponding to the cusps in $\G\ba \bH^3$ and such that the horoballs $\gamma\hpz^i$,
for $\gamma\in \Gamma$ and $1\le i\leq m$,
form a disjoint collection of open horoballs,
i.e. $\gamma \hpz^i$ intersects $\hpz^j$ if and only if $i=j$ and 
$\gamma\hpz^i=\hpz^j$.
 
We fix
\be 
\label{hor} 
\Hor:=\{ [g] \in \FM : \pi(g) \in 
\textstyle\bigcup_{\gamma,i} \gamma \hpz^i \}.
\ee
For $R\ge 0$, we  set
\be
\label{horR} 
\Hor_R=\{ [g] \in \FM : \pi(g) \in 
\textstyle\bigcup_{\gamma,i} \gamma \hpz^i_R \}.
\ee
Then each $\Hor_R$ provides a union of disjoint neighborhood of  cusps in $\RFM$ and hence  the subset 
$$\RFM-\Hor_R$$  is  compact for any $R\ge 0$.

\section{Limit circle with non-elementary stabilizer}
\label{seclimorb}
As noted before, the study of 
geodesic planes in $M=\Gamma \ba \bH^3$ can be approached in two ways: either via the study of
$H$-orbits in $\FM$ or via the study of $\Gamma$-orbits in $\mc C$.
In this section we analyze a $\Gamma$-orbit  $\Gamma C$ in $\mathcal C^*$
which accumulates on a circle $D$ with a non-elementary stabilizer in $\Gamma$.
We show that $\Gamma C$ is dense in $\mc C_\Lambda$ 
in the following two cases:
\begin{enumerate}
\item when $D$ is separating $\Lambda$ (Proposition \ref{latt});
\item when there exists a sequence of distinct circles $C_n\in \Gamma C$ such that $C_n\cap\Lambda$ 
does not collapse to a countable set (Proposition \ref{gbb}).  
\end{enumerate}
\subsection{Sweeping the limit set}
The following proposition is a useful tool, which says that,
in order to prove the density of  $\Gamma C $ in $\mc C_\Lambda$,
we only have to find a ``sweeping family'' of circles
in the closure of $\Gamma C$.

\begin{prop}  
\label{sweep}  
\cite[Corollary 4.2]{MMO} 
Let $\Gamma\subset G$ be a Zariski dense discrete subgroup. If $\mathcal D\subset \mathcal C$ is a collection of circles such that
$\bigcup_{C\in \mathcal D} C$ contains a non-empty open subset of $\Lambda$, then
there exists $C\in \mathcal D$ such that $\overline{\Gamma C}=\mathcal C_\Lambda$.
\end{prop}

In the subsection \ref{if}, we will construct such a sweeping family $\mc D$ using a result of Dalbo (Proposition \ref{da}).

In the subsequent sections, we will use other sweeping families $\mc D$
(Corollary \ref{vp}) that are constructed via
 a  more delicate polynomial divergence argument. 

\subsection{Influence of Fuchsian groups} \label{if}
Let $B\subset S^2$ be a round open disk with a hyperbolic metric  $\rho_B$.
We set
 $$G^B= \Isom^+(B, \rho_B)\simeq \PSL_2(\br).$$
A discrete subgroup of $G^B$ is a Fuchsian group, and its limit set  lies in the boundary $\partial B$.
For a non-empty subset $E\subset \partial B $,
we denote by $\hull (E, B)\subset B$ the convex hull of $E$ in $B$, and
by $\mathcal H(B, E)\subset \overline{B}$ the 
closure of the set of horocycles in $B$  resting on $E$. The only circle in $\mathcal H(B, E)$ which is not a horocycle is
$\partial B$ itself.

We first recall the following result of Dalbo \cite{Da}:

\begin{prop}
\label{da} 
Let $\Gamma \subset G^B$ be a non-elementary  discrete subgroup
with limit set $\Lambda$.
Let $C\in \mathcal H(B, \Lambda)$. If $C\cap \Lambda$ is a 
radial limit point, 
then  $\Gamma C$ is dense in 
$\mathcal H(B, \Lambda)$.
\end{prop}

The following proposition is a generalization of \cite[Corollary 3.2]{MMO} from $\Gamma$  convex cocompact to a general finitely generated group. Note that a finitely general Fuchsian subgroup is geometrically finite.

\begin{lem}
\label{dal}
Let  $\Gamma \subset G^B$ be a non-elementary finitely generated discrete subgroup.
Suppose $C_n\to \partial B$ in $\mathcal C$
and  $C_n\cap \op{hull} (\Lambda, B)\ne \emptyset$ for all $n$.
Then the closure of $\bigcup \Gamma C_n$ contains $\mathcal H(B, \Lambda)$.
\end{lem}

\begin{proof} In view of Proposition \ref{da},
it suffices to show that
  the closure of $\bigcup \Gamma C_n$ contains a circle $C\subset \overline{B}$ resting at a radial limit point of $\Gamma$.
Passing to a subsequence, we will consider two separate cases.
\vs

{\bf Case 1:
$C_n\cap \partial B =\{a_n, b_n\}\; {\rm with}\; a_n\neq b_n$.}
Note that the angle between $C_n$ and $\partial B$ converges to $0$.
Since $\Gamma $ is a finitely generated Fuchsian group, it is geometrically finite.  It follows that there exists a compacts set 
$W\subset \Gamma \ba B$ such that every geodesic 
on the surface $\Gamma \ba B$ that intersect its convex core also intersects $W$.
Applying this to each geodesic joining $a_n$ and $b_n$, we may assume 
that the sequences 
$a_n$ and $b_n$  converge to two distinct points $a_0\neq b_0$, after replacing
$C_n$ by $\delta_nC_n$ for a suitable $\delta_n\in \Gamma$ if necessary.
The open arcs $(a_n, b_n):= C_n\cap B$ converge to an open arc $(a_0, b_0)$ of $\partial B$.
We consider two separate subcases.

{\bf Case 1.A: $(a_0, b_0)\cap\Lambda \neq\emptyset$}.
Choose a hyperbolic element $\delta\in \Gamma$ whose
 repelling and attracting fixed points $\xi_r$ and $\xi_a$ lie inside the open arc $(a_0, b_0)$.
Denote by $L$ the axis of translation of $\delta$
and choose a  compact arc $L_0\subset L$ which is a fundamental set 
for the action of $\delta$ on $L$. 
For $n$ large, $C_n\cap L\ne \emptyset$. Pick  $c_n\in C_n \cap L$.
Then there exist $k_n\in \m Z$ such that $\delta^{k_n} c_n \in L_0$. Note that $k_n \to \pm\infty$ since $c_n$ converges to a point in $\partial (B)$. Hence
$\delta^{k_n}C_n$ converges to a circle contained in $\overline{B}$ resting at $\xi_r$ or at $\xi_a$.

{\bf Case 1.B: $(a_0, b_0)\cap\Lambda=\emptyset$}.
In this case, one of these two points, say $a_0$, is in the limit set 
and $(a_0,b_0)$ is included in a maximal arc $(a_0,b'_0)$ of $\partial B-\Lambda$.
Let $L$ be the geodesic of $B$ 
that connects $a_0$ and $b'_0$. 
Its projection to $\Gamma \ba B$ is included in the boundary of the surface 
$S:=\core(\Gamma\ba B)$.
Since $\Gamma$ is geometrically finite, this projection is compact,
and hence $L$ is the axis of translation of a hyperbolic element $\delta\in \Gamma$. 
For $n$ large,  $ C_n \cap L$ is non-empty, and we proceed as in 
Case 1.A.
  
{\bf Case 2:
$C_n\cap \partial B=\emptyset\; {\rm or}\; \{ a_n\}.$}
Choose any hyperbolic element $\delta\in \Gamma$ and let $L$ be the axis of translation of $\delta$.
Then  $C_n \cap L\ne \emptyset $ for all $n$, by passing to a subsequence. Hence we conclude as in 
Case 1.A. 
\end{proof}

\begin{cor}
\label{boundary}
Let $\Gamma\subset G$ be a Zariski dense discrete subgroup.
Let $B\subset S^2$ be a round open disk that meets $\Lambda$ and such that
$\Gamma^B$ is non-elementary and finitely generated.  If $C_n\to \partial B$ in $\mathcal C$
and
\be
\label{cnh}
C_n\cap \op{hull}(B, \Lambda({\Gamma^B}))\ne \emptyset ,
\ee
then the closure of $\bigcup \Gamma C_n$   contains $\mathcal C_\Lambda$.
\end{cor}

Note that Condition \ref{cnh} is always satisfied when $\Gamma^B$
is a lattice in $G^B$.
\begin{proof}
This follows from Proposition \ref{sweep} and Lemma \ref{dal}. 
\end{proof}

The following is a useful consequence of the previous discussion.

\begin{prop}
\label{latt} 
Let $\Gamma\subset G$ be a Zariski dense discrete subgroup. 
Let $C\in \mathcal C_\Lambda^*$.
Suppose that there exists a sequence of distinct circles $C_n\in\Gamma C$ 
converging to a circle $ D\in \mc C^*_\Lambda$ 
whose stabilizer  $\Gamma^D$ is non-elementary.
Then $\Gamma C$ is dense in $\mathcal C_\Lambda.$
\end{prop}

\begin{proof} 
Observe that every circle $C_n$ sufficiently near $D$
intersects  the set $\op{hull}(B, \Lambda (\Gamma^B))$
for at least one of the two disks $B$ bounded by $D$. Moreover $B\cap \Lambda\ne \emptyset$ as $D$ separates $\Lambda$ and
$\Gamma^B=\Gamma^D$ as $D$ is an oriented circle. Hence the claim follows from Corollary \ref{boundary}.  \end{proof}

\begin{cor} 
\label{latt2}
Let $\Gamma\subset G$ be a Zariski dense discrete subgroup. 
If $C\in \mathcal C^{*}_\Lambda$ has a non-elementary 
stabilizer $\Gamma^C$,
then $\Gamma C $ is  either discrete or dense in $\CL$.
\end{cor}

\begin{proof}
If the orbit $\Gamma C$ is not locally closed,  there exists a sequence
of distinct circles $C_n\in \Gamma C$ converging to $C$, and hence the claim follows from Proposition \ref{latt}.
\end{proof}

We recall that a subset is {\it locally closed} when it is open in its closure
and that a subset is  {\it discrete} when it is {\it countable and locally closed}.
Hence the conclusion of Corollary \ref{latt2} can also be stated as the dichotomy that $\Gamma C$ is either locally closed
or dense in $\CL$.

\subsection{Planes near the boundary of the convex core}
We now  explain an analog of Proposition \ref{latt}
when the limit circle $D$ is not separating $\Lambda$.
The following lemma says that its stabilizer $\Gamma^D$
is non-elementary as soon as $D\cap \Lambda$ is uncountable.
 
\begin{lem}
\label{gcf}
Let $M=\Gamma\ba \bH^3$ be a geometrically finite manifold with limit set $\Lambda$ and
$D$ be the boundary of a supporting hyperplane for $\op{hull}(\Lambda)$.
Then 
\begin{enumerate}
\item $\Gamma^D$ is a finitely generated Fuchsian group;
\item  There is a finite set $\Lambda_0$ such that
$D\cap \Lambda=\Lambda({\Gamma^D}) \cup \Gamma^D \Lambda_0$.
\end{enumerate}
\end{lem}

\begin{proof}
This lemma is stated for $\Gamma$ convex cocompact 
in \cite[Theorem 5.1]{MMO2}. But its proof works
for $\Gamma$ geometrically finite as well with no change.
\end{proof}

\begin{prop}
\label{gbb}
Let $M=\Gamma\ba \bH^3$ be a geometrically finite manifold.
Consider a sequence of circles $C_n\to D$ with $C_n\in \mathcal C_\Lambda^*$ and
$D\notin  \mathcal C_\Lambda^*$. Suppose that
$\liminf_{n\to\infty} (C_n\cap \Lambda)$ is uncountable.
Then $\bigcup \Gamma C_n$ is dense in $\CL$.
\end{prop}

\begin{proof}
This proposition is stated for $M$ convex cocompact with incompressible boundary  \cite[Theorem 6.1]{MMO2}.
With Lemmas \ref{dal} and \ref{gcf} in place,
the proof extends verbatim to the case claimed. \end{proof}

\section{$H$-orbit closures containing a periodic $U$-orbit}
\label{secperorb}
In the rest of the paper, we will use the point of view of $H$-orbits 
in the frame bundle $\FM$ in our study of geodesic planes in $M$. 
This viewpoint enables us to utilize  the dynamics of the actions of the subgroups $A$, $U$, and $V$ of $G$ on 
the space $\FM=\Gamma\ba G$.

In this section we focus on an orbit $xH$ in $F_\Lambda$
whose closure contains a periodic $U$-orbit $yU$ contained in $F^*$ and prove that
such $xH$ is either locally closed or dense in $F_\Lambda$ (Proposition \ref{cd2}).

\subsection{One-parameter family of circles}
In proving the density of an orbit $xH$ in $F_\Lambda$, 
our main strategy is to find a point $y\in F_\Lambda^*$ and 
a one-parameter semigroup $V^+\subset V$ such that
$yV^+$ is included in $\overline{xH}$ and 
to apply the following  corollary
of Proposition \ref{sweep}.

\begin{cor} 
\label{vp} 
Let $\Gamma\subset G$ be a Zariski dense discrete subgroup 
and $V^+$ a one-parameter subsemigroup of $V$.
For any $y\in F^*$, the closure
$\overline{yV^+H}$ contains $F_\Lambda$.
\end{cor}

\begin{proof}
We choose a representative $g\in G$ of $y=[g]$ and consider the corresponding circle $C=(gH)^+$. 
The union of circles $(gvH)^+$ for $v\in V^+$ contains 
a disk $B$ bounded by $C$.
Since $C$ separates $\Lambda$, 
 $B$ contains a non-empty
open subset of $\Lambda$. 
Therefore, by Proposition \ref{sweep}, 
the set $\overline{yV^+H}$ contains $F_\Lambda$.
\end{proof}

When the closure of $xH$ contains a periodic $U$-orbit $Y$,
we will apply Corollary \ref{vp} to a point $y\in Y$.

\subsection{Periodic $U$-orbits}

\begin{prop}
\label{cd2}
Let $\Gamma\subset G$ be a Zariski dense discrete subgroup
and $x\in F^*_\Lambda$. Suppose that
$\oxH\cap F^*$ contains a periodic $U$-orbit $yU$. 

Then either \begin{enumerate}
\item  $xH$ is  locally closed and $yH=xH$; or
\item  $xH$ is dense in $F_\Lambda$.
\end{enumerate}
\end{prop}

The rest of this section is devoted to a proof 
of Proposition  \ref{cd2}. We fix $x$, $y$ and $Y:=yU$ as in Proposition  \ref{cd2} and 
set $X=\overline{xH}$.

Setting
$$
S_y:=\{g\in G: yg\in  X \} ,
$$
we split the proof into the following three cases:

\begin{enumerate}
\item[$\spadesuit 1$.]$S_y\cap O\not\subset VH$ for any neighborhood $O$ of $e$;
\item[$\spadesuit 2$.] $S_y\cap O\subset VH$ for some neighborhood $O$ of $e$ and $S_y\cap O\not\subset H$ for any neighborhood $O$ of $e$;
\item[$\spadesuit 3$.] $S_y\cap O\subset H$ for some neighborhood $O$ of $e$.
\end{enumerate}
\vs

\noindent {\bf Case $\spadesuit 1$}.
In this case, we will need the following algebraic lemma:

\begin{lem}
\label{sv}
Let $S$ be a subset of $G-VH$ such that $e\in \overline S$. Then the closure of $USH$ contains a one-parameter semigroup $V^+$ of $V$.
\end{lem}

\begin{proof}  Let $\mc M_2(\m R)$ 
denote the set of real matrices of order $2$,
and consider the Lie algebra $\mc V= \begin{pmatrix} 0 & i \br \\ 0 & 0\end{pmatrix}$ of $V$.
Without loss of generality, we can assume that $S=SH$.
Then the set $S$ contains  a sequence
$s_n={\rm exp}(M_n)\to e$ with $M_n\in i\mc M_2 (\m R) - \mc V$. Since $M_n$ 
tends to $0$ as $n\to \infty$, the closure of $\bigcup_n \{u M_n u^{-1}: u \in U\}$
contains a half-line in $\mc V$. Our claim follows.
\end{proof}

\begin{lem}
\label{p1} 
In case $\spadesuit 1$,  $xH$ is dense in $F_\Lambda$.
\end{lem}

\begin{proof} The assumption implies that $S_y$ contains a subset $S$ such that $S\cap VH=\emptyset$ and
$e\in \overline{S}$. Therefore,
by Lemma \ref{sv}, the closure
of $US_y H$ contains a one-parameter semigroup $V^+$.
Hence for any $v\in V^+$, one can write 
$$
v=\lim_{n\to\infty} u_n g_n h_n
$$ 
for some $u_n\in U$, 
$g_n\in S_y$ and $h_n\in H$.
By passing to a subsequence, we may assume that $yu_n^{-1}\to y_0 \in Y$.
So $$y g_n h_n =(yu_n^{-1}) (u_n g_n h_n) \to y_0 v.$$
Hence  $y_0v$ belongs to $X$, and $yv\in y_0Uv=y_0vU$ belongs to $X$ too.
This proves the inclusion
$
yV^+\subset X .
$
Therefore, by Corollary \ref{vp}, since $y\in F^*$, the orbit 
$xH$
is dense in $F_\Lambda$.
\end{proof}
\vs

\noindent{\bf Case $\spadesuit 2$}.
In this case, we will use the following algebraic fact:

\begin{lem}
\label{alg} 
If $w\in VH$ satisfies   $wv\in VH$  for some $v\in V\!-\!\{ e\}$,
then $w\in AN$.
\end{lem}

\begin{proof}
Without loss of generality, we can assume that 
$w=
\begin{pmatrix} a & b \\ c & d\end{pmatrix}
\in H$ and  write 
$v=
\begin{pmatrix} 1 & is \\ 0 & 1\end{pmatrix}
\in V$ with $s$ real and non-zero. 
If the product $wv$ belongs to $VH$,
the lower row 
of the matrix $wv$ must be real and this implies $c=0$ as required.
\end{proof}

\begin{lem}
\label{p2} 
Case $\spadesuit 2$ does not happen.
\end{lem}

\begin{proof} 
Suppose it happens.

{\bf First step:} We claim that $xH$ contains a periodic $U$-orbit.

Indeed if $y\notin xH$, Condition $\spadesuit 2$ says that 
there exists a non-trivial $v\in V$
such that $y':=yv$ belongs to $xH$. Since $v$ commutes with $U$, the orbit $y'U\subset xH$
 is periodic.
 
 Therefore, by renaming $y'$ to $y$, $xH$ contains a periodic orbit $yU$.

{\bf Second step:}
We now claim that $yAU$ is locally closed.

If this is not the case, for any open neighborhood $O$  of $e$ in $G$, we can find $p\in AU$ and $w\in O-AU$ such that $yp=yw$.
By Condition $\spadesuit 2$, we know that $w\in VH$ and we also 
know that there exists a sequence
$v_n\to e$ in $V-\{ e\}$ such that 
$yv_n\in X$. Therefore, we get
$$
y w(p^{-1}v_np)=yv_n p\in X.
$$
If $O$ is small enough,  condition $\spadesuit 2$ implies that 
for $n$ large, 
$$
w(p^{-1}v_np)\in VH.
$$
Therefore, Lemma \ref{alg} implies that
$$
w\in AN.
$$
Hence the element 
$s:=wp^{-1}$ belongs to the stabilizer $S:=\op{Stab}_{AN}(y)$.

Let $\varphi: AN\to AN/AU=N/U$ be the natural projection.
Since the discrete group $S$ intersects $U$ cocompactly, 
it is included in $N$ and its image $\varphi(S)$ is discrete.
Since $w$ can be taken arbitrarily
close to $e$, its image $\varphi(s)=\varphi(w)$ 
can be made arbitrarily close to $e$ as well, yielding a contradiction.
{\bf Third step:} We finally claim that $xH$ is locally closed. 

Since  $yAU$ is locally closed and
 $K_H:=\SO(2)$ is compact, the set $yAUK_H=yH$ is also locally closed. 
Since $y\in xH$, we have  $xH=yH$ is locally closed.
This contradicts Condition $\spadesuit 2$. 
\end{proof}
\vs

\noindent {\bf Case $\spadesuit 3$}.
By the assumption, $yO\cap X\subset yH$ for some neighborhood $O$ of $e$ in $G$.
Hence $yH$ is open in $X$, in other words, $\oxH-yH$ is closed. Hence if $y$ didn't belong to $xH$,
then $\overline{xH}-yH$, being a closed subset containing $xH$,
should be equal to $\overline{xH}$.
Therefore $yH=xH$. Since $xH$ is open in $X$,
 $xH$ is locally closed. This finishes the proof of Proposition \ref{cd2}.

\section{Return times of $U$-orbits}
\label{secrethor}
In Section \ref{secperorb}, we have described the possible closures $\oxH$ that contain a periodic $U$-orbit.
In this section we  
recall, for each $k>1$,  a closed $A$-invariant subset $\RK\subset \RFM$
consisting of points which
have a ``thick'' set of return times to $\RK$ itself under the $U$-action; this set was introduced in \cite{MMO2}.

The main result of this section (Proposition \ref{thick}) 
is that for  the set $W_{k,R}:=\RK-\Hor_R$ where  $\Hor_R$ is a sufficiently deep cusp neighborhood,
 every point $x\in W_{k,R}$ has a thick set of return times to $W_{k,R}$ under the $U$-action.

This information will be useful for the following reasons
that will be explained in the next sections:\\
$*$ When $M$ is geometrically finite, these sets $W_{k,R}$ are compact.\\
$*$  
When $M$ is a geometrically finite acylindrical manifold,
every  $H$-orbit $xH$ in $ F^*$ intersects  $W_{k,R}$ for all sufficiently large $k>1$
(Corollary \ref{ffk}).\\
$*$  One can develop
a polynomial divergence argument for  $U$-orbits $yU$, $y\in W$,
with a thick set of return times to a compact set $W$
(Proposition \ref{ffin}).\\
$*$ This argument can be applied to a $U$-minimal subset $Y=\overline{yU}\subset \oxH$ 
 with respect to $W_{k,R}$
(Proposition \ref{main3}).

\subsection{Points with $k$-thick return times} \begin{Def}   
\label{defthi} Let $k\ge 1$.
Let  $T$ be a subset of $\m R$ and $S$
be a subset of a circle $C\subset S^2$.
 \begin{enumerate}
\item $T$ is $k$-thick at $\infty$
if
there exists $s_0\geq 0$ such that, for all $s>s_0$,
\be
\label{thi}
T\cap \left( [-ks, -s]\cup [s, ks]\right) \ne\emptyset
\ee
\item $T$ is $k$-thick
if the condition \eqref{thi} is satisfied for all $s>0$.
\item $T$ is globally $k$-thick if $T\neq \emptyset$ and  $T-t$ is $k$-thick for every $t\in T$.
\item  $S$  is  $k$-uniformly perfect
if for  any homography
$\varphi:C\to \m R\cup\{\infty\}$ such that $\varphi^{-1}(\infty)\in S$
the set $\varphi(S)\cap \m R$ is globally $k$-thick.
\item $S$ is uniformly perfect if it is $k$-uniformly perfect for some $k\ge 1$. \end{enumerate}
\end{Def}

 Let $\Gamma \subset G$ be 
a Zariski dense discrete subgroup  and $M=\Gamma\ba \bH^3$.
For $x\in \FM $, let 
$$
T_x=\{t\in \m R : xu_t\in \RFM\}
$$
be the set of return times of $x$ to the renormalized frame bundle $\RFM$ under the horocycle flow. Note that 
$T_{xu_t}=T_x-t$ for all $t\in \m R$.

Define 
\be
\label{defk}
\RK:=\{x\in \RFM: T_x\text{ contains a globally $k$-thick subset containing $0$}\}.
\ee
It is easy to check that
 $\RK$ is  a closed $A$-invariant subset such that for any $x\in \RK$,
the set $$T_{x,k}:=\{t\in \m R: xu_t\in \RK\}$$
is globally $k$-thick.
\subsection{Disjoint Horoballs}
When $\mathpzc h$ is a horoball  in $\bH^3$ and  $R$ is positive, we denote by
$\hpz_R$ the horoball contained in $\hpz$ whose distance to the boundary of $\hpz$ is $R$.
Hence the bigger $R$ is, the deeper the horoball $\hpz_R$ is.

For an interval $J\subset \br$, we denote by $\ell _J$ the length of $J$.
The following basic lemma says that ``the time spent by a horocycle
in a deep horoball $\hpz_R$ is a small fraction of the time
spent in the fixed horoball $\hpz$''.

\begin{lem} 
\label{deep} 
For any $\alpha>0$, there exists $R=R(\alpha)>0$ such that for  any horoball $\hpz$ in $\mathbb H^3$
and any $g\in G$, we have
\begin{equation}
\label{jnn}  
J_g \pm \alpha \cdot  \ell _{J_g }\subset I_g \end{equation}
where $I_g=\{t\in \br : \pi(gu_t) \in \hpz\}$ and $J_g=\{t\in \br: \pi(gu_t)\in \hpz_R\}$.
\end{lem}

\begin{proof} We use the upper half-space model of $\bH^3$.
Since 
$$
(g \hpz)_R=g\hpz_R
\;\; 
\text{ for any $g\in G$,}
$$ 
we may assume without loss of generality that
the horoball $\hpz$ is defined by $x^2+y^2+(z-h)^2\le h^2$  and that
$\pi(gU)$ is the horizontal line 
$$
L=\{ \pi(gu_t)=(t,y_0,1):t\in \br\}\subset \bH^3
$$ 
of height one.
As $I_g$ and $J_g$ are symmetric intervals, it suffices to show that
$\ell_{I_g} \ge (\alpha+1)\ell_{J_g}$
assuming that $J_g\ne \emptyset$.
The boundaries of $\hpz$ and $\hpz_R$ are defined by $x^2+y^2= 2zh-z^2$ and $x^2+y^2= 2zhe^{-R} -z^2$ respectively.
Hence the intersections $L\cap \hpz$ and $L\cap \hpz_R$ are respectively given by 
$$
t_1^2= 2h-1-y_0^2
\;\;\;\text{and}\;\;\;
t_2^2=2he^{-R}  -1-y_0^2.
$$
Note that
$\ell_{I_g}=2|t_1|$ and $\ell_{J_g}=2|t_2|$.
We compute 
$$t_2^2\leq e^{-R}t_1^2.
$$
Therefore if we take $R$ so that $e^R>(\alpha+1)^2$,
then we gets
$\ell_{I_g}\ge  (\alpha +1) \ell_{J_g}$.
\end{proof}

We will now use notations $\Hor$ and $\Hor_R$ 
as in \eqref{hor} and \eqref{horR}, 
even though we do not assume $\Gamma$ to be geometrically finite:
we just fix finitely many horoballs  $\hpz^{1}, \cdots, \hpz^m$ in $\bH^3$
such that the horoballs $\gamma\hpz^i$,
for $\gamma\in \Gamma$ and $i\leq m$,
form a disjoint collection of open horoballs and we set
\begin{equation*}
\Hor:=\{ [g] \in F : \pi(g) \in 
\textstyle\bigcup_{\gamma,i} \gamma \hpz^i \},
\end{equation*}
\begin{equation*}
\Hor_R:=\{ [g] \in F : \pi(g) \in 
\textstyle\bigcup_{\gamma,i} \gamma \hpz^i_R \}.
\end{equation*}

\begin{cor} 
\label{deep2}
Given any  $\alpha\ge 0$, there exists $R>0$ satisfying the following:  for all $x\in \FM$,
we write the set of return times of $xU$ in $\Hor$ and $\Hor_R$ as disjoint unions
of open intervals
\be
\label{injn}
\{t\in \m R :xu_t\in \Hor\}=\textstyle\bigcup I_n
\;\;{\rm and}\;\;
\{t\in \m R :xu_t\in \Hor_R\}=\textstyle\bigcup J_n
\ee 
so that $J_n\subset I_n$ for all $n$ ($J_n$ may be empty). Then \begin{equation}
\label{jnn2}  
J_n\pm \alpha \ell_{J_n}\subset I_n .\end{equation}
\end{cor}

\subsection{Thickness of the set of return times in a compact set}

The following proposition roughly says 
that for a point $x$ for which the horocyclic flow returns often
to $\RK$, 
its $U$-flow also returns often to the set $$W_{k,R}:=\RK-\Hor_R.$$

\begin{prop} 
\label{thick} 
Let $\Gamma \subset G$ be a
Zariski dense discrete subgroup. 
Then, there exists $R>0$  such that  for any $x\in W_{k, R} $,
the set 
$$\{t\in \m R: xu_t\in W_{k, R} \}$$ is $4k$-thick at $\infty$.
\end{prop}

\begin{proof} 
Writing $x=[g]\in \RK$, we may assume 
$g^+=\infty$ and $g^-=0$ without loss of generality. 
This means that $gU$ is a horizontal line.
It follows from the definition of $\RK$ that the set
$$
T_{x,k} := \{t\in \m R: xu_t \in \RK\}.
$$
is globally $k$-thick. Let $R=R(2k)$ be as given by Lemma \ref{deep}.
We use again the decomposition  \eqref{injn} of the sets of return times 
as a union of disjoint open intervals:
$\{t\in \br: xu_t\in \Hor\}=\bigcup I_n$ 
and
$\{t\in \m R: xu_t\in \Hor_R\}=\bigcup J_n$ with $J_n\subset I_n$.
Since $xU$ is not contained in $\Hor_R$
the  intervals $J_n$ have finite length.
Write $\ell_n=\ell_{J_n}$.   
It follows from Corollary \ref{deep2} that for each $n$,
\begin{equation}
\label{jn}  
J_n\pm 2k  \ell_{n}\subset I_n. 
\end{equation}
If $x\not\in\Hor$, set $s_x=0$. 
If $x\in\Hor$, let $n$ be the integer such that $0\in I_n=(a_n, b_n)$
and set $s_x=\max (|a_n|, |b_{n}|) $.
 
We  claim that  for all $s> s_x$,
\be
\label{te} 
\left( T_{x,k} -\textstyle\bigcup J_n\right)  
\cap  \left( [-4ks, -s]\cup [s, 4 ks] \right) \ne \emptyset .\ee
This means that
$ \{u\in U: xu\in \RK-\Hor_R\}$ is $4k$-thick at $\infty$.

Since $T_{x,k}$ is  $k$-thick, there exists 
$$t\in T_{x,k} \cap ([-2ks, -2s] \cup [2s, 2ks]).$$
If $t\notin J_n$ for some $n$, the claim \eqref{te} holds.
Hence suppose $t\in J_n$. 
By the choice of $s_x$, we have $$0\notin I_n .$$
 
By \eqref{jn} and by the fact that $I_n$ does not contain $0$,
we have $$2 k \ell_n \le |t| .$$
 
By the global $k$-thickness of $T_{x,k}$,
there exists 
$$
t'\in T_{x,k}\cap (t\pm [\ell_n, k\ell_n]).
$$

By \eqref{jn},
we have $t'\in I_n$. Clearly, $t'\notin J_n$.
In order to prove \eqref{te}, it remains to prove $s\le| t'| \le 4ks$.
For this, note that $|t| -k\ell_n \le |t'|\le |t| +k\ell_n$. 
Hence 
\begin{equation*}
|t'| \le  2|t| \le 4ks\quad \text{ and } \quad |t'|\ge |t|/2 \ge s,
\end{equation*}
proving the claim. \end{proof}

Note that in the above proof, we set $s_x=0$ if $x\notin \Hor$. Therefore we have the following corollary of the proof:
\begin{cor}
for any $x \in\RK-\Hor$, 
the set $$\{t\in \m R: xu_t\in \RK-\Hor_R \}$$ is $4k$-thick.
\end{cor}

\section{Homogeneous dynamics}
\label{sechomdyn}
In this section we explain the polynomial divergence argument
for Zariski dense subgroups $\Gamma$.
The main assumption requires that the set of return times of the horocyclic flow in a suitable 
compact set is $k$-thick at $\infty$. 
\vs


We begin by the following lemma which is analogous to Lemmas \ref{sv} and \ref{alg}
and that we will apply to a set $T$ of return times 
in the proof of Proposition \ref{ffin}.

\begin{lem} 
\label{thickT} 
Let $T\subset U$ be a subset which is $k$-thick at $\infty$.
\begin{enumerate}
\item If $g_n\to e$ in $G-VH$, then 
$\displaystyle\limsup_{n\to\infty} T g_n H$ contains a sequence $v_n\to e$ in $V-\{e\}$.
\item If $g_n\to e$ in $G-AN$, then 
$\displaystyle\limsup_{n\to\infty} T g_n U$ contains a sequence $\ell_n\to e$ in $AV-\{e\}$.
\end{enumerate}
\end{lem}

\begin{proof}
Lemma \ref{thickT} is a slight modification of \cite[Thm 8.1 and 8.2]{MMO} where
it was stated for a sequence $T_n$ of $k$-thick subsets instead of a single set $T$.
If all $T_n$ are equal to a fixed $T$,
the $k$-thickness at $\infty$ is sufficient for the proof.
\end{proof}

For a compact subset $W$ of $\FM=\Gamma\backslash G$,
a $U$-invariant closed subset $Y\subset \Gamma\ba G$ is said to be 
{\it $U$-minimal 
with respect to  $W$} if
$Y\cap W\ne \emptyset$ and $\overline{yU}=Y$ for any $y\in Y\cap W$. Such a minimal set $Y$ always exists.

\begin{prop}
\label{ffin}  Let $\Gamma\subset G$ be a Zariski dense discrete subgroup with limit set $\Lambda$.
Let $W\subset F_\Lambda^* $ be a compact subset, and $X\subset \FM$ a closed $H$-invariant subset intersecting $W$.
Let $Y\subset X$ be a $U$-minimal subset with respect to  $W$. Assume that
\begin{enumerate}
\item
there exists $k\ge 1$ such that for any $y\in Y\cap W$, 
$
T_y=\{t\in \br: yu_t\in W\}
$ 
is $k$-thick at $\infty$;
\item 
there exists $y\in Y\cap W$ such that $X-yH$ is not closed;
\item
there exists $y\in Y\cap W$ such that $yU$ is not periodic;
\item
there exist $y\in Y\cap W$  and $t_n \to +\infty$ such that
$ya_{-t_n}$ belongs to $ W$.
\end{enumerate}
Then $$ X=\FL.$$ 
\end{prop}

The proof of this proposition is based on  the polynomial divergence argument, which allows us to find an orbit $yV^+$
sitting inside $X$. Combined with Corollary \ref{vp}, 
this implies that $X$ is dense. 

\begin{proof}  
{\bf First step}: We claim that
there exist $v_n\to e$ in $V\!-\!\{ e\}$ such that $Yv_n\subset X$.

We follow the proof of \cite[Lem. 9.7]{MMO}. Let $y\in Y\cap W$ be as in (2). Since $yH$ is not open in $X$,
  there exists a sequence $g_n\to e$ in $G-H$ 
such that $yg_n\in X$.

If $g_n\in VH$, the claim follows easily. 

If $g_n\notin VH$, then,
by Lemma \ref{thickT},
there exist $t_n \in  T_y$, 
$h_n \in H$ such that  
$u_{-t_n} g_n h_n \to v$ for some arbitrarily small $v\in V-\{e\}$.
Since $yu_{t_n}\in Y\cap W$ and $Y\cap W$ is compact, the sequence $yu_{t_n}$
converges to some $y_0\in Y\cap W$, by passing to a subsequence. It follows that the element
$y_0 v=\displaystyle\lim_{n\to\infty} yg_nh_n$ is contained in $X$, and hence $Yv\subset X$ by the minimality. Since $v$ can be taken arbitrarily close to $e$, this proves our first claim.
\vs

{\bf Second step}:
We claim that there exists a one-parameter semigroup $L\subset AV$ such that  $YL\subset Y$. 
 Let $y\in Y\cap W$ be as in (3). 
We follow the proof of \cite[Theorem 9.4, Lemma 9.5]{MMO}.
By (1),
there exists a sequence $t_n\in T_y$ tending to $\infty$ such that 
$yu_{t_n} \to y_0\in Y\cap W$.
Write $yu_{t_n}=y_0 g_n $ where $g_n \to e$. Since $yU$ is not periodic,  $g_n\in G-U$.

If $g_n\in AN$, then the closed semigroup generated by these $g_n$'s contains a 
one-parameter semigroup $L$ as desired. 

If $g_n\notin AN$,
then, by Lemma \ref{thickT}, 
there exist $s_n \in  T_{y_0}$, $ u'_n \in U$ 
such that  $u_{-s_n} g_n u'_n \to \ell $ 
for some non-trivial $\ell \in AV$. 
Since  $y_0u_{s_n}\in Y\cap W$ and $Y\cap W$ is compact, passing to a subsequence, the sequence $y_0 u_{s_n}$ 
converges to some $y_1\in Y\cap W$.
It follows that the point 
$$y_1\ell=\displaystyle\lim_{n\to\infty}  y_0g_nu'_n
=\lim_{n\to\infty}  yu_{t_n}u'_n$$
belongs to $Y$, and hence 
$Y\ell \subset Y$ 
by the minimality of $Y$.  As $\ell$ can be taken arbitrarily close to $e$, this yields
the desired one-parameter semigroup $L$  and proves our second claim. 
\vs

{\bf Third step}: We claim that there exists an interval $V_I$ of $V$ containing $e$ such that $ YV_I\subset X$. 

By an interval $V_I$ of $V$ we mean an infinite connected subset of $V$.

We follow the proof of \cite[Theorem 7.1]{MMO2}. We use the second step.
A one-parameter semigroup $L\subset AV$ is either a semigroup $V^+\subset V$, a semigroup $A^+\subset A$
or $v^{-1} A^+v $ for some non-trivial $v\in V$.

If $L=V^+$  our third claim is clear.

If $L=v^{-1} A^+v $, we have $Y  v^{-1} A^+v A \subset X$. 
Now the set $ v^{-1} A^+v A$ contains 
an interval  $V_I$ of $V$ containing $e$.
This proves our third claim.

If $L=A^+$, this semigroup is not transversal to $H$ and 
we  need to use also the first step. 
Indeed we have $$Y(\bigcup A^+v_n A) \subset X .$$
Now the closure of $\bigcup A^+v_nA$ contains an interval $V_I$ 
of $V$ containing $e$. 
This proves our third claim.
\vs

{\bf Fourth step}: We claim that there exist a one-parameter semigroup $V^+\subset V$ and $w\in W$ such that $wV^+\subset X$. 

Here $V^+$ is a one-parameter semigroup of $V$ 
intersecting $V_I$ as a non-trivial interval. Let $y, t_n$ be as in (4).
By the compactness of $W$, 
the sequence $ya_{-t_n}$ converges to a point $w\in X\cap W$, by passing to a subsequence.
For every $v\in V^+$, 
there exists a sequence $v_n\in V_I$ such that $a_{t_n}v_na_{-t_n}=v$ for $n$ large.
It follows that 
$$wv=\displaystyle\lim_{n\to\infty}  yv_na_{-t_n}\in X$$ and hence 
$wV^+\subset X$.
This proves the fourth claim.

Now, since $W\subset F_\Lambda^*$,  
Corollary \ref{vp} implies that $X=F_\Lambda$.
\end{proof}

\section{$H$-orbits are locally closed or dense}
\label{seccloden}
In this section, we explain how the polynomial divergence argument can be applied 
to the orbit closure of $xH$ for  $x\in \RF_k M\cap F^*$,
when $\Gamma$ is geometrically finite.
The main advantage of the assumption that $\Gamma$ is geometrically finite is that
 $\RK-\Hor_R$ is compact.

The key result of this section is the following theorem:

\begin{thm}
\label{rt}  
Let $\Gamma\subset  G$ be a geometrically finite Zariski dense subgroup.  
If $x\in \RK\cap F^*$ for some $k\ge 1$, then
 ${xH}$ is  either locally closed
or dense in $\FL$. 
\end{thm}

In Theorem \ref{rt} the assumptions on $\Gamma$ are very general,
but we point out that the conclusion 
concerns only those $H$-orbits intersecting the set $\RF_k M\cap F^*$.

Fix $R\geq 1$ as given by Proposition \ref{thick}
and set 
$$
W^*_{k,R}:=(\RF_k M -\Hor_R)\cap F^*.
$$

\begin{prop} 
\label{outside} 
Let $\Gamma\subset  G$ be a geometrically finite Zariski dense subgroup.  
Let $x\in \RK\cap F^*$. If $\overline{xH}\cap W^*_{k,R}$ is not compact,
then  $xH$ is dense in $F_\Lambda$.
\end{prop}

\begin{proof}
This follows from Proposition \ref{gbb}; see the proof of \cite[Coro. 6.2]{MMO2}.
\end{proof}

\begin{prop} 
\label{main3} 
Let $\Gamma\subset  G$ be a geometrically finite Zariski dense subgroup.  
Let $x\in \RK\cap F^*$. If $\overline{xH}\cap \FS$ contains no periodic $U$-orbit
and   $(\oxH - xH)\cap  W^*_{k,R}$ is non-empty, then 
$xH$ is dense in $\FL$. 
\end{prop}

\begin{proof}
Using Proposition \ref{outside}, we may assume that 
the set $\oxH\cap W^*_{k,R}$ is compact.
By assumption,   the set
$(\overline{xH}-xH)\cap W^*_{k,R}$ is non-empty.

We introduce the compact set
\be 
\label{defw}
W=
\begin{cases}  
&(\oxH-xH)\cap W^*_{k,R} \text{\;\; if $xH$ is locally closed}\\
&\oxH\cap W^*_{k,R} \text{\hspace{10ex} if $xH$ is not locally closed.}
\end{cases}
\ee

Let $Y\subset \oxH$ be a $U$-minimal subset with respect to $W$.
We want to apply Proposition \ref{ffin}. We check that its four assumptions are satisfied:\\
1. By Proposition \ref{thick}, 
for any $y\in Y\cap W$, the set 
$
T_y:=\{t\in \m R: yu_t\in Y\cap W\}
$ is $4k$-thick at $\infty$.\\
2. We can  find  $y\in Y\cap W$ such that  $\oxH -yH$ is not closed. To see this,
if $Y\cap W\not\subset xH$, we choose $y\in Y\cap 
W-xH$. If $\overline{xH}-yH$ were closed, it would be a closed subset containing $xH$, contradicting $y\in X$.

If $Y\cap W\subset xH$, by \eqref{defw}, the orbit $xH$ is not locally closed and we choose any $y\in Y\cap W$.
\\3.
The $U$-orbit of any point $y\in Y\cap W\subset \oxH\cap \FS$ is never periodic.\\
4.  Since  $T_y$ is uncountable 
while $\Lambda_p$ is countable, there exists $y_0:=[g_0]\in yU\cap W$ such that
 $g_0^{-}$ defined in \eqref{gpm}
is a radial limit point.
Since 
$g_0^{-}$ is a radial limit point, there exists a sequence $t_n\to + \infty$
such that $y_0a_{-t_n}\not\in \Hor_R$. Since both
$\RK\cap F^*$ and $xH$ are $A$-invariant, 
we have $y_0a_{-t_n}\in W$.

Hence Proposition \ref{ffin} implies that
$xH$ is dense in $\FL$.
\end{proof}

\begin{proof}[Proof of Theorem \ref{rt}] 
When  $\oxH\cap F^*$ contains a periodic $U$-orbit, the claim follows from Proposition \ref{cd2}.

We now assume that   $\oxH\cap F^*$ contains no periodic $U$-orbit. 
Suppose that $xH$ is not dense in $\FL$. 
Then, by Proposition \ref{main3}, the set 
\be
\label{wxh}
W:=xH\cap W_{k,R}^*
\ee
is compact.
We now repeat almost verbatim the same proof of Proposition \ref{main3}, with this compact set $W$. 
Let $Y\subset \oxH$ be a $U$-minimal subset relative to $W$.
We assume that the orbit $xH$ is not locally closed.
The four assumptions of Proposition \ref{ffin} are still valid:\\
1. For any $y\in Y\cap W$, $\{t: yu_t\in W\}$ is $4k$-thick at $\infty$.\\ 
2. For any $y\in Y\cap W$, the set $\oxH-yH$ is not closed.\\
3.  For any $y\in Y\cap W$,   $yU$ is not periodic.\\
4.  Choose any $y=[g]\in Y\cap W$  such that $g^-$ is a radial limit point. Then $ya_{-t_n}\in W$ for some $t_n\to \infty$.
\\ 
Hence, by  Proposition \ref{ffin}, the orbit
$xH$ is dense in $\FL$. Contradiction.\end{proof}

\section{Locally closed $H$-orbits have non-elementary stabilizer}
\label{secnonele}
In this section we give more information on locally closed $H$-orbits intersecting the set $\RK\cap \F^*$. In particular
 we show that they have 
non-elementary stabilizer. 
\vs

As in the previous section, let $k\geq 1$ and
fix $R\geq 1$ as given by Proposition \ref{thick}
and set 
$
W^*_{k,R}:=(\RK - \Hor_R)\cap F^*.
$
The main aim of this section is to prove the following:

\begin{thm} 
\label{ns} 
Let $\Gamma\subset G$ be a geometrically finite  Zariski dense subgroup.
Suppose that $xH\subset \FM$ is a locally closed subset intersecting
 $\RK\cap F^*$.  Then
\begin{enumerate}
\item
 $ \oxH\cap( \RK ) H\cap F^*\subset xH$, i.e., $xH$ is closed in $( \RK )H\cap F^*$;
\item
The stabilizer $H_x$ of $x$ in $H$ is non-elementary.
\end{enumerate}
In particular, there are only countably many 
locally closed $H$-orbits intersecting $\RK\cap F^*$.
\end{thm}

In Theorem \ref{ns} the assumptions on $\Gamma$ are very general,
but we  again point out that the conclusion 
concerns only those orbits intersecting $x\in \RK\cap F^*$.

We will apply the following lemma with $L=U$.

\begin{lemma}
\label{gy1} 
Let $L$ be a one-parameter group acting continuously on $\FM$, and 
$Y\subset \FM $ be an $L$-minimal subset 
with respect to a compact subset $W\subset \FM $. 
Fix $y\in Y\cap W$ and suppose that 
$\{\ell\in L : y\ell\in Y\cap W\}$ is unbounded.
Then there exists a sequence 
$\ell_n \to \infty$ in $L$ such that $y\ell_n\to y$.
\end{lemma}

\begin{proof} The set $Z:=\{z\in Y:\exists \ell_n\to \infty
\text{ in } L \text{ such that } y\ell_n\to z\}$ of $\omega$-limit points of the orbit $yL$ 
is a closed $L$-invariant subset of $Y$ that intersects $W$.
Therefore, by minimality, it is equal to $Y$ and hence contains $y$.
\end{proof}

We will also need the following lemma:

\begin{lem}
\label{ih} 
For any $R\ge 0$, the set $\Hor_R$ never contains an $A$-orbit.\\
In particular, if $xH$ intersects $\RK$, 
it also intersects $ W_{k,R}$.
\end{lem}

\begin{proof}  The claim follows because
no horoball in $\bH^3$ contains a complete geodesic and the set $\RK$ is $A$-invariant.
\end{proof}

Fix $R\geq 1$ as given by Proposition \ref{thick}.
The condition that $xH$ is locally closed implies that
the orbit map  
\be
\label{proper}
H_x \ba H \to xH
\ee 
given by $[h]\mapsto xh$ is a proper map
when $xH$ is endowed with the induced topology from $\FM $ (cf. \cite{Zi}).

\begin{proof}[Proof of Theorem \ref{ns}]
 Let $xH$ be a locally closed set with $x\in \RK\cap F^*$.

{\bf First case :}  $\oxH\cap F^*$ contains no periodic $U$-orbit.

(1) In this case, since $xH$ is locally closed,  
Propositions  \ref{outside} and  \ref{main3} 
imply that the intersection $W:=\oxH\cap W^*_{k,R}$ is compact and is contained in $xH$. 
Since this is true for any $R$ large enough, the 
intersection $\oxH\cap \RK\cap F^*$ is also contained in $xH$.

(2) As in the proof of 
Theorem \ref{rt}, we will use the following compact set 
$$
W=xH\cap W^*_{k,R}.
$$ 
By Lemma \ref{ih}, $W$ is non-empty. 

We first construct a sequence of elements $\delta_n$ in the stabilizer $H_x$. As in the proof of Theorem \ref{rt}, we introduce
a $U$-minimal subset $Y$ of $\overline{xH}$ with respect to $W$, 
and fix a point 
$$
y\in Y\cap W\subset xH.
$$ 
We may assume $y=x$ without loss of generality.
By Proposition \ref{thick}, $\{t\in \m R : xu_{t} \in Y\cap W\}$ is unbounded.
Hence by Lemma \ref{gy1}, one can find  a sequence  
$u_n\to \infty$ in  $U$ such that $yu_n \to y$.
By the properness of the map in \eqref{proper},
there exists  a sequence $\delta_n \to\infty$ in $H_x$ such that
\begin{equation}\label{du} \delta_n  u_n \to e \quad  \text{ in $H$}.\end{equation}

Let $J\subset H$ denote the Zariski closure of
$H_x$. We want to prove  $J=H$.

We first claim that $U\subset J$.
Since the homogeneous space $J\ba H$ is real algebraic, 
any $U$-orbit in $J\ba H$ is locally closed \cite[3.18]{BT}.  Therefore, since the sequence $[e]u_n$ converges to $[e]$ in $J\ba H$,
the stabilizer of $[e]$ in $U$ is non trivial and hence $ U \subset J$.

We now claim that  $J\not \subset AU$.   
Indeed, if the elements $\delta_n$ were in $AU$ 
one would write $\delta_n=a_n u'_n$ with $a_n\in A$ and $u'_n\in U$. 
Since $xH$ contains no periodic $U$-orbit, the stabilizer
$H_x$ does not contain any unipotent element, and hence $a_n \ne e$.
Since the sequence $a_n u'_n u_n$ converges to  $e$, the sequence $a_n$
also converges to  $e$. Therefore, $\delta_n$ 
is a sequence of  hyperbolic elements of a discrete subgroup of $AU$ whose
eigenvalues go to $1$. Contradiction.  
 
Since any algebraic subgroup of $H$ 
containing $U$ but not contained in $AU$
is only $H$ itself, we obtain  $J=H$. 
Therefore $ H_x$ is Zariski dense in $H$. 
\vs

{\bf Second case :}   $\oxH\cap F^*$ contains a periodic $U$-orbit $yU$.

(1) By Proposition \ref{outside}, the set $W:=\oxH\cap W_{k,R}$ is compact.
We claim that  $W$  is contained in $xH$. If not, the set $(\oxH-xH) \cap W^*_{k,R}$ would contain an element $x'$.
Since $xH$ is locally closed, it is not included in 
$\overline{x'H}$. By Proposition \ref{cd2}, any periodic $U$-orbit of $\oxH$ is contained in $xH$. Therefore 
\be
\label{xhw}
\overline{x'H}
\text{ contains no periodic $U$-orbit.}
\ee
Since $x'H$ can not be dense in $F_\Lambda$,  Theorem \ref{rt} 
implies that  $x'H$ is locally closed. 
By the first case considered above,
the stabilizer $H_{x'}$ is non-elementary.
Therefore, since $x'\in F_\Lambda^*$, Proposition \ref{latt} implies that $xH$ is dense in $F_\Lambda$. Since $xH$ is locally closed, this is a contradiction.
Therefore we obtain:
$$
W=xH\cap W^*_{k,R}.
$$
Since this is true for any $R$ large enough, the 
intersection $\oxH\cap (\RK\cap F^*)$ is contained in $xH$.

(2)  Since $xH$ is locally closed, by Proposition \ref{cd2}, 
$yU$ is contained in $xH$.
This proves that the stabilizer $H_x$ contains a non-trivial unipotent element.

We now construct a  non-trivial hyperbolic element $\delta\in H_x$. 
We write $x=[g]$ and denote by $C=(gH)^+$ the corresponding circle.
We use again the compact set $W$ given by \eqref{xhw}. 
Since $x$ is in $\RK$, the set 
$$
\{(gh)^+\in C: h\in H\text{ with } xh \in W\}
$$ is uncountable  and hence
contains at least two radial limit points. 
By considering an $A$-orbit which connects these two limit points,
we get $z\in xH $ such that $zA\subset xH\cap  (\RK\cap F^*)$. 
Since the set $\{t\in \m R: za_t\in \Hor_R\}$ is a disjoint union of  bounded open intervals,  
we can find a sequence $a_n\in A$ such that $za_n\in W$ and
$a'_n:=a_n^{-1} a_{n+1}\to \infty$.

Since $xH\cap W$ is compact, 
by passing to a subsequence the sequence 
$z_n:=za_n$ converges to a point
$z'\in xH\cap W$. 

Therefore we have
$$
z_n\to z'
\;\;\text{ and}\;\;
z_na'_n=z_{n+1}\to z'.
$$
By the properness of the map in \eqref{proper},
we can write $z_n =z'\varepsilon_n$ with $\varepsilon_n\to e$ in $H$.
Therefore the following product belongs to the stabilizer of $z'$~:
\begin{equation}
\label{daa} 
\delta_n:=  \varepsilon_n a'_n \varepsilon_{n+1}^{-1} \in H_{z'}.
\end{equation}
Since $a'_n\to\infty$,  the elements $\delta_n$ are non-trivial hyperbolic
elements of $H_{z'}$ for all $n$ large enough.
Since $z'\in xH$, 
the stabilizer $H_x$ also contains non-trivial hyperbolic elements.

Since a discrete subgroup of $H$ containing
simultaneously non-trivial unipotent elements
and non-trivial hyperbolic elements is non-elementary,
the group $H_x$ is non-elementary.

If $x=[g]$, then
$gHg^{-1}\cap \Gamma$ is non-elementary, and hence contains a Zariski dense finitely generated subgroup
of $gHg^{-1}$. 
Now  the last claim on the countability  follows because there are only countably many finitely generated subgroups of $\Gamma$ and
$H$ has  index two in its normalizer. \end{proof}

\section{A uniformly perfect subset of a circle} 
\label{secuniper}
In this section we give an interpretation of Theorems \ref{rt} and \ref{ns} 
in terms of $\Gamma$-orbits of circles.
\vs

The union 
$\RF_\infty M:=\cup_{k\geq 1}\RK$ can be described as
$$
\RF_\infty M=\{[g]\in \RFM: \text{$(gH)^+\cap \Lambda\supset S\supset \{g^\pm\}$ for a
uniformly perfect set $S$}\}.
$$
Hence an $H$-orbit $[g]H$ intersects 
$\RF_\infty M$ if and only if the intersection  $C\cap \Lambda$
contains a uniformly perfect subset $S$ where $C$ is the circle given by $(gH)^+$.

Putting together Theorems \ref{rt} and \ref{ns}, we obtain:
\begin{cor}
\label{rt22}  
Let $\Gamma\subset  G$ be a geometrically finite Zariski dense subgroup.  
If $x\in \RF_\infty M \cap F^*$, then
 ${xH}$ is  either locally closed
or dense in $\FL$. When $xH$ is locally closed, it is closed in $(\RF_\infty M)H\cap F^*$ and $H_x$ is non-elementary.
\end{cor}

Let $\mc C_{\rm perf}$ denote the set of all circles $C\in \mc C$ such that $C\cap \Lambda$ contains a 
uniformly perfect subset of $C$, and let 
$$\mc C_{\rm perf}^*:=\mc C_{\rm perf}\cap\mc C^*.$$
This is a $\Gamma$-invariant set and in general, this set is neither open nor closed in $\mc C$.
The following theorem is equivalent to Corollary \ref{rt22}.

\begin{thm} 
\label{circperf} 
Let $\Gamma\subset G$ be a geometrically finite Zariski dense subgroup.
For any $C\in \mathcal C_{\rm perf}^{*}$, the orbit 
$\Gamma C$ is either discrete or dense in $\mathcal C_\Lambda$.
Moreover, a discrete orbit $\Gamma C$ is closed in  
$\mathcal C_{\rm perf}^*$ and its stabilizer $\Gamma^C$ is non-elementary.
\end{thm}

\section{$H$-orbits contained in $\RFM$}
\label{secorbcon}
In this section, we prove Theorem \ref{m4}.
We assume that $\Gamma\subset G$ is geometrically finite and 
of infinite co-volume.  
We describe the $H$-orbits $xH$ which are contained in the renormalized frame bundle $\RFM$. 
Equivalently we describe the $\Gamma$-orbits $\Gamma C$ contained in the limit set $\Lambda$. 

\begin{thm}
\label{hr}
Let $\Gamma\subset G$ be a geometrically finite subgroup 
of infinite co-volume and let $x\in \FM$.
\begin{enumerate}
\item If  $xH$ is contained in $ \RFM$, then  $xH$ is closed and has finite volume.
\item If  $xH\subset \FM$ has finite volume then $xH$ is closed and contained in $\RFM$.
\item There are only finitely many $H$-orbits contained in $\RFM$.
\end{enumerate}
\end{thm}

This theorem implies that the closed subset 
$$
 F_0:=\{ x\in F : xH\subset \RFM\}
$$
is a union of finitely many closed $H$-orbits.
Equivalently,  the closed subset 
$$
\mc C_0:=\{C\in \mc C: C\subset \Lambda\}
$$
is a union of finitely many closed $\Gamma$-orbits.

We will deduce Theorem \ref{hr} from the following three lemmas.
We write $x=[g]$ and denote by $C$ the corresponding circle
$C:=(gH)^+$.

\begin{lem}
\label{loc}  Let $\Gamma\subset G$ be a geometrically finite subgroup 
of infinite co-volume.
Any $xH$  included in $\RFM$  is  closed.
\end{lem}

\begin{proof}  Remember that, by assumption, the circle
$C$ is contained in the limit set $\Lambda$. Let $F_0^*:= F_0\cap F^*$.

{\bf Case 1:} $x\not\in F^*_0$. In this case the circle 
$C$ is a boundary circle.
This means that 
$C$ is included in $\Lambda$
and bounds a disk $B$ in $\Omega=S^2-\Lambda$. 
The surface 
$ \Gamma^C\ba\op{hull}(C)$ is then a connected component of the boundary 
of the core of $M$. Hence the orbit $xH$ is closed.

{\bf Case 2:} $x\in F^*_0$. 
By Definition  \ref{defk}, we have 
$$
F_0\subset \RK
$$
for all $k\geq 1$.
Since $x$ belongs to $\RK\cap F^*$,  Theorems \ref{rt} and \ref{ns}  
implies that either the orbit $xH$ is dense in $F_\Lambda$
or this orbit is closed in $F_0^*$.

Since $\Gamma$ is geometrically finite of infinite covolume, 
the limit set $\Lambda$ is not the whole sphere $S^2$ and therefore 
$\RFM$ is strictly smaller than $F_\Lambda$. 
Therefore the orbit $xH$ can not be dense in $F_\Lambda$.

Hence  $xH$ is closed in $F_0^*$. If $xH$ were  not closed in $F_0$,
there would exist an $H$-orbit $yH\subset \overline{xH}\subset F_0$ which does not lie in $F^{*}_0$.
By Case 1, the orbit $yH$ corresponds to a boundary circle 
and hence it is of finite volume. 
Then by Corollary \ref{boundary}, 
the orbit $xH$ has to be dense in $F_\Lambda$. Contradiction. 
\end{proof}

\begin{lemma}
\label{fv}
Let $\Gamma\subset G$ be a geometrically finite subgroup.
Any closed $xH$ contained in $\RFM$ has finite volume.
\end{lemma}

\begin{proof}
Since $xH$ is closed, the inclusion map $xH\to \RFM$ is  proper, and the corresponding map $\Gamma^C\ba\op{hull}(C)\to  \op{core}(M)$
is also  proper.  Since $\op{core}(M)$ has finite volume,
using explicit descriptions of $\op{core}(M)$ in the neighborhood  of both
rank-one  cusps and rank-two cusps as in 
\cite[\S 3]{Ma1}, we can deduce that 
the surface $ \Gamma^C\ba\op{hull}(C)$ has finite area.
\end{proof}

\begin{lem}\label{iso}
Let $\Gamma\subset G$ be a discrete group with limit set $\Lambda\ne S^2$. 
Then the following set $\mc C_a$ is discrete in $\mc C$~: 
$$
\mc C_a:=\{C\in \mathcal C: \Gamma^C\ba\op{hull}(C) \text{ is of finite area}\} .
$$
\end{lem}

\begin{proof}
Suppose $C\in \mc C_a$ is not isolated. Then there exists a sequence of distinct circles $C_n\in\mc C_a$ converging to $C$. Since $\Lambda(\Gamma^C)=C$, by Corollary \ref{boundary}, the closure of $\bigcup \Gamma C_n$ contains $\mathcal C_\Lambda$.  
On the other hand, any circle in this closure is contained in $\Lambda$.
This contradicts the assumption $\Lambda\ne S^2$. 
\end{proof}

\begin{proof}[Proof of Theorem \ref{hr}]
(1) follows from Lemmas \ref{loc} and \ref{fv}.

(2) Since the quotient $\Gamma^C\ba\op{hull}(C)$ has finite area,
the limit set $\Lambda(\Gamma^C)$ is equal to $C$ and 
the orbit $xH$ is contained in $\RFM$. By (1), $xH$ is  closed.

(3) By (1) and (2) we have the equality $\mc C_0=\mc C_a$. Therefore by Lemma \ref{iso},
the set $\mc C_0$ is closed and discrete. 

Now since $\Gamma$ is geometrically finite, by Lemma \ref{ih}, 
any $H$-orbit contained in $\RFM$ intersects the compact set $W:=\RFM-\Hor$.
Hence there exists a compact set $\mc K\subset \mc C$
such that, for any circle $C\in \mc C_0$, there exists $\gamma$ in $\Gamma$
such that $\gamma C$ is in  $\mc K$. 
Since the intersection $\mc C_0\cap \mc K$ is discrete and compact, it is finite.
Therefore $\Gamma$ has only finitely many orbits in $\mc C_0$,
and there are only finitely many $H$-orbits  in $F_0$.
\end{proof}

\section{Planes in geometrically finite acylindrical manifolds}
\label{secacyman}

In this section,
we assume that $M=\Gamma\ba \m H^3$ is a geometrically finite  acylindrical  manifold,
and we prove Theorems \ref{m1} and \ref{m2}.
Under this assumption, every separating $H$-orbit $xH$ intersects 
$\RK\cap F^*$ and hence $W_{k, R}^*$ for $k$ large enough (Corollary \ref{ffk}) which enables us to apply
Theorems \ref{rt} and \ref{ns}.
\vs

\subsection{Geometrically finite acylindrical manifolds} We begin with the definition of an acylindrical manifold.
 Let $D^2$ denote a closed $2$-disk and let $C^2=S^1\times [0,1]$ be a cylinder.
A compact $3$-manifold $N$ is called {\it acylindrical}
\begin{enumerate}
\item if the boundary of $N$ is incompressible, i.e., any continuous map $f:(D^2, \partial D^2)\to (N, \partial N)$
can be deformed  into $\partial{N}$ or equivalently if the inclusion $\pi_1(R)\to \pi_1(N)$
is injective for any component $R$ of $\partial N$; and
\item if any essential cylinder of $N$ is boundary parallel, i.e., any
continuous map  $f:(C^2, \partial C^2)\to (N, \partial {N})$, injective on $\pi_1$, can be deformed into
  $\partial{N}$.
\end{enumerate}

Let $M=\Gamma\ba \bH^3$ be a geometrically finite manifold, and consider
the associated Kleinian manifold
$$\overline{M}=\Gamma\ba (\bH^3\cup \Omega).$$

Recall that  a compact connected submanifold of $\op{Int} (\oM)\simeq M $ is called a core of $M$ if
the inclusion $\pi_1(N)\to \pi_1(M)$ is an isomorphism and each component of $\partial N$ is the full boundary of a non-compact component of $\op{Int}(\oM)-N$ \cite[3.12]{Ma1}. This is sometimes called the Scott core of $M$
and is unique up to a homeomorphism.

We say that a  geometrically finite manifold $M$ is {\it acylindrical} if its compact core $N$ is acylindrical (\cite{Th}, \cite[\S 4.7]{Ma1})

\begin{Rmk}\rm
The Apollonian manifold $\mathcal A\ba \bH^3$ is not acylindrical, because its compact core is a handle body of genus $2$, and hence it is not boundary incompressible.
\end{Rmk}

For the rest of this section, we let
$$\text{$M=\Gamma\ba \bH^3$ be a geometrically finite acylindrical manifold of infinite volume.}$$
It is convenient to choose an explicit model of the core of $M$. As $M$ is geometrically finite,
$\overline M$ is compact except for a finite number of rank one and rank two cusps.
The rank one cusps correspond to  pairs of punctures on $\partial \overline M$ which are arranged 
 so that each pair determines a solid pairing tube, and the rank two cusps determine solid cusp tori \cite[\S 3]{Ma1}. 
 We let $N_0$ be the core of $M$ which is obtained by removing the  interiors of all solid pairing tubes and solid cusp tori which can be chosen to be mutually disjoint.

  So $N_0$ is a compact submanifold of $\op{Int} \oM\simeq M$ whose boundary consists of finitely many closed surfaces of genus at least $2$
  with marked paring cylinders  and torus boundary components.

Write the domain of discontinuity $\Omega=S^2-\Lambda$ as a union of components:
$$\Omega=\cup_i B_i .$$

When all $B_i$ are round open disks, $M$ is called rigid acylindrical.

Let $\Delta_i$ denote the stabilizer of $B_i$ in $\Gamma$. 

A quasi-fuchsian group  is a Kleinian group which leaves a Jordan curve invariant.
A finitely generated quasi-fuchsian group whose limit set is the invariant Jordan curve is a quasi-conformal deformation of a lattice of $\PSL_2(\br)$ and
hence its limit set is a quasi-circle.
\begin{lemma} \label{ddd}
\begin{enumerate}
\item  For each $i$, $\Delta_i$ is a finitely generated quasi-fuchsian  group and $B_i$ is an open Jordan disk
with $\partial B_i=\Lambda(\Delta_i)$.

\item There are countably infinitely many $B_i$'s with finitely many $\Gamma$-orbits and $\Lambda=\overline{\bigcup_i \partial B_i}$.
\item For each  $i\ne j$, $\overline{B_i}\cap \overline{B_j}$ is either empty or a parabolic fixed point of rank one;
and any rank one parabolic fixed point  of $\Gamma$ arises in this way.
\item No subset of $\{\overline B_i\}$ forms a loop of tangent disks, i.e.,  if $B_{i_1}, \cdots, B_{i_\ell}$, $\ell \ge 3$ is a sequence of distinct disks such that
$\overline B_{i_j}\cap \overline B_{i_{j+1}} \ne \emptyset$ for all $j$, then $\overline{B_{i_{\ell}}}\cap \overline{B_{i_1}}=\emptyset$.

\end{enumerate}
\end{lemma}
\begin{proof} Since $N_0$
is boundary incompressible, 
$\Lambda$ is connected. This implies that
each $B_i$ is simply connected. By Ahlfors \cite[Lemma 2]{Ah}, $\Delta_i$ is finitely generated and $B_i$ is a component of $\Omega(\Delta_i)$.
 Since $\Delta_i$ has two invariant components,
it is quasi-fuchsian and  $\partial B_i=\Lambda(\Delta_i)$ (cf. \cite[Theorem 3]{Mas}).
For (2), since $\Gamma$ is not quasi-fuchsian by the acylindricality assumption, there are countably infinitely many $B_i$'s. There are only finitely many
$\Gamma$-orbits of $B_i$'s by Ahlfors' finiteness theorem. 
Since  ${\bigcup_i \partial B_i}$ is a $\Gamma$-invariant subset of  $\Lambda$, the last claim in (2) is clear.

 For (3), note that   as $\Delta_i$ is a component subgroup of $\Gamma$, 
we have by \cite[Theorem 3]{Mas2}, \be\label{one} \Lambda(\Delta_i\cap \Delta_j)=\Lambda(\Delta_i)\cap \Lambda(\Delta_j).\ee

 On the other hand,  as $N_0$ is acylindrical,  a loxodromic element can preserve at most one of $B_i$'s.   As no rank two parabolic fixed point
 lies in $\bigcup_i\partial B_i$, it follows that for all $i\ne j$,
 $\Delta_i\cap \Delta_j$ is either trivial or the stabilizer of a parabolic limit point of $\Gamma$ of rank one.
 Therefore  the first claim in (3) follows from (1) and \eqref{one}. The second claim in (3) follows from a standard fact about a geometrically finite group
 \cite[\S 3]{Ma1}.

To prove (4), suppose not. Then 
one gets a non-trivial  loop on $\partial N_0$
which bounds a disk in $N_0$, contradicting the boundary-incompressible condition on $N_0$.
\end{proof}

\subsection{Lower bounds for moduli of annuli and Cantor sets}
The closures of the components of $\Omega$ may intersect with each other, but we can regroup the components of $\Omega$
into maximal trees of disks whose closures are not only disjoint but are uniformly apart in the language of moduli.

Fix a closed surface $S$ of $\partial N_0$ of genus at least $2$. Its fundamental group $\pi_1(S)$ injects to $\Gamma$
as $S$ is an incompressible surface.  By the construction of $N_0$,
 $S$ comes with finitely many marked cylinders $P_i$'s and each connected component of $S-\cup P_i$ corresponds to a component
 of $\partial \overline M$.   In view of $\partial \overline M =\Gamma \ba \Omega$, 
those components of $\Omega$
invariant by the image of $\pi_1(S)$ in $\Gamma$ can be described using the following equivalence relations: 
we let $B_j\sim B_k$ if their closures intersect each other. This spans an
 equivalence relation on the collection $\{B_i\}$.
 We write $$\Omega =\bigcup T_\ell$$ where $T_\ell$ is the union of all disks in the same equivalence class, i.e., $T_\ell$ is a maximal tree of disks.

The $T_\ell$'s fall into finitely many $\Gamma$-orbits corresponding to non-toroidal closed surfaces of $\partial N_0$, and
the stabilizer of ${T_\ell}$ in $\Gamma$ is conjugate to
$\pi_1(S_\ell)\subset \Gamma$ for a non-toroidal closed  surface $S_\ell$ of $\partial N_0$ and
has infinite index in $\Gamma$ (cf. \cite{Ab2}, \cite{Fl},  \cite{Wa}).
We denote by $\Gamma_\ell$ the stabilizer of $T_\ell$ in $\Gamma$.

The following lemma is a crucial ingredient of Theorem \ref{sier}.
\begin{lemma} \label{each} We have:
\begin{enumerate}
\item For each $\ell$, $\partial T_\ell=  \Lambda (\Gamma_\ell) $.
\item  For each $\ell$,  $\overline T_\ell$ is connected, locally contractible and simply connected.
\item For each $k\ne \ell$,
 $ \overline T_k\cap \overline T_\ell=\emptyset$. \end{enumerate}
\end{lemma}
\begin{proof}
Let $K:=\overline{T_\ell}$.
 and consider the
set $\mathfrak \{p_{ij}=\overline{B_i}\cap \overline{B_j}: i\ne j, B_i, B_j\subset T_\ell\}$.
Each $p_{ij}$ is fixed by a parabolic element, say, $\gamma_{ij}\in \Gamma_\ell$. 

By the uniformization theorem,
$\pi_1(S_\ell)$ can be realized as a cocompact lattice $\Sigma_\ell$ in $\PSL_2(\br)$ acting on $\bH^2$ as isometries. 
Let $c_{ij}\in \bH^2$ be the geodesic
stabilized by a hyperbolic element $\iota(\gamma_{ij})$ of $\Sigma_\ell$
for the isomorphism $\iota: \Gamma_\ell \simeq \Sigma_\ell$.

 Let $K_0$ be the compact set  obtained from 
 $\overline{\bH^2}$ by collapsing geodesic arcs $c_{ij}$ to single points.
We denote by $\partial K_0$ the image of $S^1$ in this collapsing process, so that
$K_0$ is an "abstract tree of disks" while $\partial K_0$ is an "abstract tree of circles".
As $\Gamma_\ell$ is  a finitely generated subgroup
of a geometrically finite group $\Gamma$, it itself is geometrically finite by a theorem of Thurston (cf. \cite[Theorem 3.11]{MT}).  So we can apply
a theorem of Floyd \cite{Fl}  to $\Gamma_\ell$ to obtain a continuous equivariant surjective map  $\psi_\ell :S^1=\partial \bH^2 \to \Lambda(\Gamma_\ell)$ conjugating $\Sigma_\ell$ to
  $\Gamma_\ell$,  which is $2$ to $1$ onto rank one parabolic fixed points of $\Gamma_\ell$ and injective everywhere else.
  Moreover, this map $\psi_\ell$ factors through the map $S^1\to \partial K_0$ described above, 
since each $p_{ij}$ is a rank one
parabolic fixed point of $\Gamma_\ell$.
It follows that $\Lambda(\Gamma_\ell) $ is equal to the closure of ${\cup_{B_j\subset T_\ell} \partial B_j}$, and hence
 $\Lambda(\Gamma_\ell)=\partial K$ follows as claimed in (1).

Denote by $\psi_\ell^*:\partial K_0\to \partial K$ the continuous surjective map induced by  $\psi_\ell$.
 We claim that $\psi_\ell^*$ is injective by the acylindrical condition on $N_0$.
Suppose not.
Then there exists  a  point $p\in \Lambda(\Gamma_\ell) - \cup _j\partial B_{j}$
over which $\psi_\ell^*$ is not injective. By the aforementioned Floyd's theorem, $p$
 is a fixed point of a parabolic element $\gamma\in \Gamma_\ell$ and 
the geodesic in $\bH^2$ whose two end points are mapped to $p$ by $\psi_\ell$
is stabilized by a hyperbolic element of $\Sigma_\ell$ corresponding to $\gamma$. Hence
this gives rise to  a non-trivial closed curve $\alpha$
in $S_\ell \subset \partial N_0$ which is homotopic to a boundary component $\beta$ of a pairing cylinder
 in $\partial N_0$.  Since $p\notin \cup_j \partial B_j$, $\alpha$ and $\beta$ are not homotopic within $S_\ell$, yielding an essential cylinder
 in $N_0$ which is not boundary parallel.
This contradicts the acylindrical condition on $N_0$.

Therefore $\psi_\ell^*$ is injective.  
Now the map $\psi_\ell^*:\partial K_0\mapsto \partial K$
can be extended to a continuous injective map  $ K_0 \mapsto K$ 
inducing homeomorphisms between each $\overline B_{0,j}$ and $\overline B_j$ where 
 $B_{0,j}$ denote the connected component
of $K_0 - \partial K_0$ above $B_j$.
Hence $K_0$ and $K$ are homeomorphic. This implies the claim (2).

   For (3), note that for all $\ell \ne k$, $$\Gamma_\ell \cap \Gamma_k =\{e\}$$ by the acylindrical condition on $N_0$.
   
  By \cite[Theorem 3.14]{MT}, we have
  $$\Lambda(\Gamma_\ell )\cap \Lambda(\Gamma_k)=\Lambda(\Gamma_\ell \cap \Gamma_k)\cup P$$
where $P$ is the set of $\zeta\in S^2-\Lambda(\Gamma_\ell \cap \Gamma_k)$ such that
$\op{Stab}_{\Gamma_\ell}\zeta$ and $\op{Stab}_{\Gamma_k}\zeta$ generates a parabolic subgroup of rank $2$ and 
$\op{Stab}_{\Gamma_\ell }\zeta \cap\op{Stab}_{\Gamma_k}\zeta =\{e\}$. 
  If $\zeta\in P$, then
  $\zeta$, being a parabolic fixed point of $\Gamma_\ell$, must arise as the intersection $\overline{B_i}\cap \overline B_j$ for some $i\ne j$
  by Floyd's theorem as discussed above. But
   every such point is a rank one parabolic fixed point of $\Gamma$ by Lemma \ref{ddd}(3); so $P=\emptyset$.
Hence
we deduce that $\Lambda(\Gamma_\ell )\cap \Lambda(\Gamma_k)=\emptyset$, which implies the claim (3)  by the claim (1).
\end{proof}

An annulus $A\subset S^2$ is an open region whose complement consists of two components.
\begin{cor}\begin{enumerate}
\item For each $\ell$, $S^2-\overline{T_{\ell}}$ is a disk.
\item For each $\ell \ne k$, $S^2 -(\overline{T_{\ell}}\cup \overline{T_k})$ is an annulus.
\end{enumerate}
\end{cor}
 
 \begin{proof} We may assume without loss of generality that $\overline{T_{\ell}}$ contains $\infty$ in $\hat{\c}=S^2$.
Now by (2) of Lemma \ref{each}, and by Alexander duality (cf. \cite[\S 3, Theorem 3.44]{Ha}),
$U:=S^2-\overline{T_{\ell}}$ is a connected open subset of $\c$.
By  the Riemann mapping theorem \cite[\S 6.1]{Ah1}, a connected open subset of $\c$
whose complement in $S^2$ is connected is a disk. As $\overline{T_{\ell}}$ is connected, it follows that $U$
is a disk, proving (1).

The claim (2) follows from (1) and  Lemma \ref{each}(3). \end{proof}

 When neither component of $S^2-A$ is a point, an annulus $A$ is conformally equivalent to a unique round annulus
 $\{z\in \c: 1<|z|<R\}$, and its modulus $\mod (A)$ is defined to be $\frac{\log R}{2\pi}$ \cite{LV}.

  \begin{thm} \label{sier} There exists $\delta>0$ such that
 \begin{equation} \label{pm} \inf_{\ell \ne k} \mod (S^2 -(\overline{T_\ell}\cup \overline{T_k}))\ge \delta .\end{equation}
\end{thm}
\begin{proof}
Suppose that the claim does not hold. Since there are only finitely many $\Gamma$-orbits on $T_\ell$'s, and the $\Gamma$-action is conformal,
 without loss of generality, we may assume that there exists $T_{k_0}$ and $T_{\ell_0}$
 such that for some infinite sequence $T_\ell \in \Gamma(T_{\ell_0})$,
  $ \mod (S^2 -(\overline{T_{k_0}}\cup \overline{T_\ell }))\to 0$ as $\ell \to 0$.
  For ease of notation, we set $k_0=1$.
 
Consider the disk $V_1:=S^2-\overline {T_1}$. 
 Since $\Gamma_1$ is the stabilizer of $T_1$ with $\Lambda(\Gamma_1)=\overline T_1 -T_1$,  $\Gamma_1$ acts on $V_1$ properly discontinuously. Hence
by the uniformization theorem together with the fact that
 $\Gamma_1$ is conjugate to $\pi_1(S)$ of a non-toroidal closed surface $S$ of $\partial N_0$,
$V:=\Gamma_1\ba V_1$ is a closed
 hyperbolic surface with $\pi_1(V)=\Gamma_1$.

Since $T_{\ell}$'s have disjoint closures, each $T_{\ell}$ maps injectively into $V$. We  denote  its image by $T_{\ell}'$.

We claim that
the diameter of $T_{\ell}'$, measured in the hyperbolic metric of $V$, goes to $0$ as $\ell \to \infty$. 
As $T_{\ell}\in \Gamma (T_{\ell_0})$ by the assumption,
 we may write $T_{\ell}=\delta_\ell(T_{\ell_0})$ for $\delta_\ell \in \Gamma$.
Since $\Gamma_1$ acts cocompactly on $V_1$, there is a compact fundamental domain, say $F$ in $V_1$, such that $\gamma_\ell(\partial T_\ell) \cap F\ne \emptyset$ for some $\gamma_\ell\in \Gamma_1$. Suppose that the diameter of $T_\ell$ in the hyperbolic metric of $V_1$
 does not tend to $0$. Then by the compactness of $F$,
 up to passing to a subsequence, $\gamma_\ell (\partial T_{\ell})=\gamma_\ell \delta_\ell (\partial T_{\ell_0})$ converges to a closed
 set $L\subset \Lambda$  in the Hausdorff topology where
  $L$ has  positive diameter
 and intersects $F$ non-trivially.  In particular, $ \op{hull}(L)\cap \bH^3\ne \emptyset$.
 Fixing $x_0 \in \op{hull}(L)\cap \bH^3$, we have a sequence $s_\ell\in \op{hull}(\partial T_{\ell_0})\cap \bH^3$ such that  $\gamma_\ell\delta_\ell (s_\ell) \to x_0$. Moreover, we can find $\gamma_\ell' \in \Gamma_{\ell_0}$ so that $( \gamma_\ell ')^{-1} s_\ell$ belongs to a fixed compact subset, say, $K$, of $\bH^3$.
 Indeed, if the injectivity radius of $x_0$ is $\epsilon>0$, then all $s_\ell$ should lie in the $\epsilon/2$-thick part of the convex hull of $\Lambda(\Gamma_{\ell_0})$ on which $\Gamma_{\ell_0}$ acts cocompactly as it is geometrically finite.

 Hence $\gamma_\ell \delta_\ell \gamma_\ell' (K) $ accumulates on a neighborhood of $x_0$. As $\Gamma$ acts properly discontinuously on $\bH^3$, this means that $\{\gamma_\ell \delta_\ell \gamma_\ell'\}$ is a finite set, and hence
 $\gamma_\ell T_\ell=\gamma_\ell\delta_\ell \gamma_\ell' T_{\ell_0}$ is a finite set.
 It follows that, up to passing to a subsequence, for all $\ell$, $\gamma_\ell T_\ell$ is a constant sequence, containing a point in $F$.
  As $T_\ell$'s are disjoint, $\gamma_\ell\in \Gamma_1$ must be an infinite sequence.
 On the other hand, if $\gamma_\ell$ is an infinite sequence, $T_\ell \cap \gamma_\ell^{-1} F\ne \emptyset$ implies that
 $\overline T_\ell\cap \Lambda(\Gamma_1)\ne \emptyset$, yielding a contradiction.
This proves the claim. 

Now if the diameter of $T_\ell'$ is smaller than one quarter of the injectivity radius of $V$, then
 we can find a disk $D_\ell$ in $V$ containing $T_\ell'$ whose diameter is $1/2$ of the injectivity radius of $V$;
 this gives a uniform lower bound, say, $\delta_0$ for the modulus of $ S^2 -(\overline{T_1}\cup \overline{T_\ell})$ for all $ \ell\ge \ell_0$ for some $\ell_0>1$,
 contradicting our assumption. This proves the claim.
 \end{proof}

If $K$ is a Cantor set of a circle $C$, we say $K$ has modulus $\epsilon$ if
\be \label{cp} \inf_{i\ne j} \mod (S^2-(\overline I_i\cup \overline I_j))\ge \epsilon\ee
where $C-K=\bigcup I_j$ is a disjoint union of intervals with disjoint closures.
Note that a Cantor set $K\subset C$ has a positive modulus if and only if $K$ is a uniformly perfect subset as in Definition \ref{defthi}.

\begin{thm}  \label{sep} \label{cm}
Let $M$ be a geometrically finite acylindrical manifold of infinite volume. 
Then for any $C\in \CS$,
$C\cap \Lambda$ contains a Cantor set of modulus $\delta'$
where $\delta'>0$ depends only on $\delta$ as in \eqref{pm}. \end{thm}
\begin{proof}
This can be proved by a slight adaptation of the proof of \cite[Theorem 3.4]{MMO2}, in view of Theorem \ref{sier}. 
Recalling $\Omega=\bigcup_{\ell =1}^\infty T_\ell$,
we write $U:=C-\Lambda=\cup U_\ell$ where $U_\ell =C\cap T_\ell$. 
Note that distinct $U_\ell$ have disjoint closures.
We may assume that $U$ is dense without loss of generality. We will say that an open interval $I=(a,b)\subset C$
with $a\ne b$ is a bridge of type $\ell$ if $a, b\in \partial U_\ell$.
We can then construct a sequence of bridges $I_j$ with disjoint closures
 precisely in the same way 
 as in proof of \cite[Theorem 3.4]{MMO2}.  We first assume that
$|I_1|>|C|/2$ by using a conformal map $g\in G^C$. We let $I_2$ be a bridge of maximal length among all those which are disjoint from $I_1$
and of a different type from $I_1$. We then enlarge $I_1$ to a maximal bridge of the same type disjoint from $I_2$.
For $k\ge 3$, we proceed inductively to define $I_k$ to be a bridge of maximal length among all bridges disjoint from $I_1, \cdots, I_{k-1}$.
We have $|I_1|\ge |I_2|\ge |I_3|\cdots$,   
$|I_k|\to 0$ and $\cup I_k$ is dense in $C$. Now $K:= C-\cup I_k$ is a desired Cantor set; pick two indices
$i<j$.

If $I_i$ and $I_j$ have 
same types, 
then $\mod (S^2-(\overline I_i\cup \overline I_j))\ge \delta'$ for some universal constant $\delta'>0$, using the fact that there must be a bridge $I_k$ with $1<k<i$
such that $I_1\cup I_k$ separates $I_i$ from $I_j$.

Now if $I_i$ and $I_j$ have
different types, say $\ell $ and $k$, 
then as the annulus $S^2-(\overline T_\ell \cup \overline T_k)$ separates $\partial I_i$ from $\partial I_j$,
we have by \cite[Ch II, Thm 1.1]{LV},
$$\mod (S^2-(\overline I_i\cup \overline I_j))\ge \mod (S^2-(\overline T_\ell \cup \overline T_k))\ge \delta$$
proving the claim.
\end{proof}

\subsection{Applications to $H$-orbits}
The following corollary follows directly from Theorem \ref{sep} and Lemma \ref{ih}. 

\begin{cor} 
\label{ffk}
Let $M=\Gamma\ba \bH^3$ be a geometrically finite acylindrical manifold.
Then for all sufficiently large $k>1$,
$$
F^* \subset (\RF_k M)H.
$$
In particular, every $H$-orbit in $F^*$
intersects $W_{k, R}$ for any $R\ge 0$. 
\end{cor}

\begin{thm} 
\label{mainacy} 
Let $M=\Gamma\ba \bH^3$ be a geometrically finite acylindrical manifold.
\begin{enumerate}
\item For  $x\in \FS$,  $xH$ is either closed in $F^*$ or dense in $F_\Lambda$.

\item If $xH$ is closed in $F^*$, the stabilizer $H_x=\{h\in H: xh=x\} $ is non-elementary.

\item There are only countably many closed $H$-orbits in $\FS$.
\item Any $H$-invariant subset  in $F^*$ is either a union of
 finitely many closed $H$-orbits in $F^*$, or
is dense in $F_\Lambda$.
\end{enumerate}
\end{thm}

\begin{proof}  Claim (1) is immediate from Corollary \ref{rt22} and Corollary \ref{ffk}.

The statements (2) and (3)  follow from Theorem \ref{ns} 
combined with Corollary \ref{ffk}.

To prove (4), let $E\subset F^*$ be an $H$-invariant subset, and
let $X$ be the closure of $E$. Suppose $X$ is not $F_\Lambda$. Then
every $H$-orbit in $E$ is closed in $F^*$. We claim that there are only finitely many $H$-orbits in $E$. Suppose not. Then
 by Theorem \ref{mainacy},
the set $E$ consists of infinitely and countably many closed $H$-orbits in $F^*$, 
so $E=\bigcup_{i\geq 1} x_iH.$
We claim that $X=F_\Lambda$, which would yield a contradiction.
By Corollary \ref{ffk} and Lemma \ref{ih}, we may assume 
that all  $x_i$ belong to $W_{k,R}^*$ for all  sufficiently large $k>1$ and $R>0$. 
As in Proposition \ref{outside}, the set  $X\cap W_{k,R}^*$ is compact, and hence
the sequence $x_i$ has a limit, say, $x \in X\cap W_{k,R}^*$.
We may assume $x$ does not lie in $E$
by replacing $E$ by a smaller subset if necessary.
Hence $E$ is not closed. Now $(X-E)\cap W_k^*$ is non-empty, and
we can repeat the proof of Proposition \ref{main3}  to get $X=F_\Lambda$.
\end{proof}

Theorem \ref{m2} is a special case of the following
theorem:

\begin{thm} 
\label{mainacy2} 
Let $M=\Gamma\ba \bH^3$ be a geometrically finite rigid acylindrical  manifold.
For any $x\in \FS$, the orbit $xH$ is either closed or dense in $\FL$.

\end{thm}
\begin{proof}
We suppose $xH$ is not closed in $F $. By Theorem \ref{mainacy}, 
there exists an $H$-orbit $yH\subset \overline{xH}$ which does not lie in $F^{*}$.
We write $y=[g]$ and set $C=(gH)^+$ to be the corresponding circle. Then the circle $C$ is 
contained in the closure of a component of the domain of discontinuity $\Omega$ 
and is tangent to its boundary $\partial \Omega$. Hence, 
by Proposition \ref{da} 
(due to Hedlund \cite{He} in this case), the closure 
$\overline{yH}$ contains a compact $H$-orbit $zH$. 
By Corollary \ref{boundary}, this forces $\overline{xH}=F_\Lambda$. 
\end{proof}

\section{Arithmetic and Non-arithmetic manifolds}\label{sarith}
In this section, we will prove Theorem \ref{arith1} and present a counterexample to it
when $M_0$ is not arithmetic.

A good background reference for this section is \cite{MR}.

\subsection{Proof Theorem \ref{arith1}.}
Recall $G=\PSL_2(\c)$ and $H=\PSL_2(\br)$.
A real Lie group $\tilde G\subset \GL_N(\br)$ is said to be a real algebraic group defined over
 $\q$ if there exists a polynomial ideal $I\subset \q [a_{ij}, \text{det}(a_{ij})^{-1}]$ in $N^2+1$ variables such that
$$\tilde G=\{(a_{ij}) \in \GL_N(\br): p(a_{ij}, \text{det}(a_{ij})^{-1})=0\text{ for all $p\in I$}\} .$$

Let $\Gamma_0 \subset G$ be an arithmetic subgroup. This means that
there exist a semisimple real algebraic group $\tilde G\subset \GL_N(\br)$ defined over $\q$ and
 a surjective real Lie group homomorphism $\psi: G\to \tilde G$  with compact kernel such
that $\psi(\Gamma_0)$ is commensurable with $\tilde G\cap \GL_N(\z)$ (cf. \cite{Bor}).



Let $\Gamma \subset \Gamma_0$ be a geometrically finite acylindrical group, and denote by $p:\Gamma\ba G \to \Gamma_0\ba G$ the canonical
projection map.
In order to prove Theorem \ref{arith1},  we first note that the orbit $xH$
is either closed in $F^*$ or dense in $F_\Lambda$ by Theorem \ref{mainacy}.

Suppose that $xH$ is closed  in $F^*$.
Write $x=[g]$ and let $C:=(gH)^+$ be the corresponding circle. 
Then,
by Theorem \ref{mainacy}, the stabilizer
$\Gamma^C$ of $C$ in $\Gamma$ is a non-elementary subgroup  
and hence is Zariski dense in $G^C$.
Hence the group $\psi(G^C)$ is defined over $\q$ and
the group $\psi(\Gamma_0^C)$ is  commensurable with $\psi(G^C)\cap \GL_N(\z)$.
By Borel and Harish-Chandra's theorem  \cite[Corollary 13.2]{Bor}, 
  $\Gamma_0^C$ is a lattice in $G^C$.
Hence the orbit $p(x)H$ has finite volume and is closed in $\Gamma_0\ba G$. 

When $x H$ is dense in $F_\Lambda$, 
its image $p(x)H$ is dense in $p(F_\Lambda)$.
On the other hand, $F_\Lambda$ has non-empty interior as $\Lambda$ is connected.
Therefore $p(F_\Lambda)\subset \Gamma_0\ba G$ is an $A$-invariant subset with non-empty interior.
Since $\Gamma_0\ba G$ contains a dense $A$-orbit, it follows that
the set $p(F_\Lambda)$ is dense in $\Gamma_0\ba G$.
Therefore the image
$p(x)H$ is also dense in $\Gamma_0\ba G$. The other implications are easy to see.

\subsection{Cutting finite volume hyperbolic manifolds}\label{cut}
In the rest of this section, we will construct an example of a {\it non-arithmetic} manifold $M_0$ and a rigid acylindrical manifold $M$ which covers $M_0$ for which Theorem \ref{arith1}
fails, as described in Proposition \ref{na}.

We  first explain how to construct
a geometrically finite {\it rigid} acylindrical manifold $M$  starting from a hyperbolic manifold $M_0$ of finite volume.
We also explain the construction of  an arithmetic hyperbolic manifold $M_0$ which admits a properly immersed  geodesic surface.
Let $\Gamma_0$ be a lattice in $G$ such that 
 $\Delta:=H\cap \Gamma_0$ 
is a cocompact lattice in $H$.
This gives rise to a properly immersed compact geodesic surface $S_0=\Delta\ba \bH^2$ in the orbifold  $M_0=\Gamma_0\ba \bH^3$.
According to \cite[Theorem 5.3.4]{MR}, by passing to a subgroup of finite index in $\Gamma_0$, we may assume that $M_0$
is a manifold and that $S_0$ is
  properly embedded
in $M_0$. 

We cut $M_0$ along $S_0$. 
The completion $N_0$ of a connected component, say, $M'$, of the complement $M_0-S_0$ is a hyperbolic
$3$-manifold whose boundary $\partial N_0$ is totally geodesic  and is the union
of one or two copies of $S_0$. 

Note that the fundamental group $\Gamma$ of $N_0$ can be considered as a subgroup of $\Gamma_0$.
Indeed, let $p_0:\m H^3\to M_0$ be the natural projection 
and $E_0$ be the closure of a connected component of $p_0^{-1}(M')$.
Since $E_0$ is convex,  $\Gamma$ 
can be identified with the stabilizer of $E_0$ in $\Gamma_0$. 
The complete  manifold $M=\Gamma\ba \bH^3$ is a
geometrically finite manifold whose convex-core boundary is isometric to $\partial N_0$, and which does not have any rank one cusp.
The domain of discontinuity $\Omega$ is a dense union of round open disks whose closures are mutually disjoint, that is,
$\Lambda$ is a Sierpinski curve. It follows that
 $M$ is a rigid  geometrically finite acylindrical  manifold.

Conversely,  any rigid acylindrical 
geometrically finite manifold $M$ is obtained that way. Indeed, the double $M_0$ of the convex core $N_0$
of $M$ along its totally geodesic boundary $\partial N_0$ is a finite volume hyperbolic $3$-manifold. 

\subsection{Comparing closures of geodesic planes} \label{go}
Let $$q=q(x_1, x_2, x_3, x_4)$$ be a quadratic form  with real coefficients and signature $(3,1)$. The hyperbolic space can be seen as
\be
\label{hqa}
\m H^3\simeq \{[v]: v\in \m R^4 \; ,\; q(v)<0\} 
\ee
where $[v]$ denotes the real line containing $v$,
and the group $G $ is isomorphic to the identity component  $\SO(q)^\circ$  of the special orthogonal group.
When $v\in \m R^4$ is such that $q(v)>0$, 
the restriction $q|_{v^\perp}$ 
is a quadratic form of signature $(2,1)$ where $v^\perp$ denotes the orthogonal complement to $v$ with respect to $q$.
Therefore 
$$
\m H_v:=v^\perp\cap\m H^3
$$ 
is a totally geodesic plane in $\m H^3$ and
the stabilizer $G_v= \SO(q|_{v^\perp})^\circ$ is isomorphic to $\PSL_2(\br)$. 
Hence the space $\mc C$ of circles $C\subset S^2$ 
or, equivalently, of
geodesic planes in $\m H^3$, can be seen as
$$
\mc C\simeq\{\m H_v: v\in \m R^4 \; ,\; q(v)>0\}.
$$

We now assume that the coefficients of the quadratic form $q$ belong to a totally real number field $\m K $
of degree $d$
such that for each non-trivial embedding $\sigma$ of $\m K$ in $\m R$,
 the quadratic form $q^\sigma$ has signature $(4,0)$ or $(0,4)$, so that the orthogonal group $\SO(q^{\sigma})$ is compact.

Now any arithmetic subgroup $\Gamma_0\subset G$ 
which contains a cocompact Fuchsian group
 is commensurable with
$$
G\cap \SL_4(\m O)
$$ 
where $G=\SO(q)^\circ$ and $\m O$ is the ring of integers of $\m K$ (cf. \cite[Section 10.2]{MR}).
In this case, the corresponding finite volume hyperbolic orbifold $M_0=\Gamma_0\ba \bH^3$ is  an {\it arithmetic} $3$-orbifold
with a properly immersed totally geodesic arithmetic surface. 
This orbifold is compact if and only if $q$ does not represent $0$ over $\m K$.

We introduce the set of rational positive lines:
$$
\mc C_\m K :=\{[v]: v\in \m K^4 \; ,\; q(v)>0\} ,$$
which can also be thought as the set of rational planes
$\{\m H_v: v\in \m K^4 \; ,\; q(v)>0\}$.

For $v\in \mc C_\m K$, the restriction $q|_{v^\perp}$ 
is  defined over $\m K$.
The group $G_v(\m O):=G_v\cap \SL_4(\m O)$ is an arithmetic subgroup of $G_v$. We call $P_v:=G_v(\m O)\ba \bH^2$ an arithmetic geodesic surface of $M_0$.
We recall that every properly embedded  geodesic plane $Q$ of $M_0$
is arithmetic, i.e. we have $Q=P_v$ for some $[v]\in \mc C_\m K$.

By passing to a finite cover, we assume $P_v$ is properly imbedded in $M_0$, and
let $\Gamma\subset \Gamma_0$ be the fundamental group of a component $M_0-P_v$ as discussed in the subsection \ref{cut}.
\label{secarithm}

\subsection{Arithmetic examples}
We begin by a very explicit example of arithmetic lattice with $\m K=\m Q$. 
We consider the following family of quadratic forms
$$q_a:= q_0(x_1, x_2, x_3) +ax_4^2 $$
where $q_0(x_1, x_2, x_3)=7 x_1^2+7x_2^2-x_3^2$ and  $a$ is a positive integer.

One can check, using the Hasse-Minkowski principle \cite[Theorem 8 in Chapter 4]{Ser}, that
\begin{enumerate}
\item $q_0$ does not represent $0$ over $\q$;
\item $q_a$ represents $0$ over $\q$ iff $a$ is a square mod $7$ and $a\ne -1$ mod $8$;
\end{enumerate}
We choose a positive integer $a$
satisfying condition (2). For instance  $a=1$ or $a=2$ will do.
We denote by $\Gamma_{a,0}:=\SO(q_a,\m Z)^\circ$ the corresponding arithmetic lattice and 
by $M_{a,0}=\Gamma_{a,0}\ba\m H^3$ the corresponding arithmetic hyperbolic orbifold.

We introduce the geodesic plane $\m H_{v_0}$ associated to 
the vector
$$
v_0=(0,0,0,1).
$$
Note that,
seen in the model \eqref{hqa} of $\m H^3$ given by the quadratic form $q_a$,
the geodesic plane  $\m H_{v_0}$ does not depend on $a$.
We fix the geodesic orbifold $S_{a,0}:=P_{v_0}$. 
Passing as above to a finite cover, we may assume that 
$M_{a,0}$ is a hyperbolic manifold and that  $S_{a,0}$ 
is a properly embedded compact geodesic surface.
After cutting along $S_{a,0}$, we get a hyperbolic $3$-manifold $N_{a,0}$ whose 
geodesic boundary $\partial N_{a,0}$ is a union of one or two copies of 
$S_{a,0}$. We denote by $\Gamma_a$ the fundamental group of $N_{a,0}$ and by 
$$
M_a:=\Gamma_a\ba \m H^3
$$ 
the corresponding geometrically finite
rigid acylindrical 
hyperbolic manifold whose convex core has compact boundary. 
As we have seen, the group $\Gamma_a$ is naturally a subgroup of $\Gamma_{a,0}$
and there is a natural covering map
$$
p_a:M_a\to M_{a,0}
$$
to which Theorem \ref{arith1} applies.

\subsection{Non-arithmetic examples}
The class of non-arithmetic hyperbolic $3$-manifolds  that we will now construct 
are some of those
introduced by Gromov and Piatetski-Shapiro in \cite{GP}.

Choose two positive integers $a,a'$ such that $a/a'\notin \q^2$, for instance $a=1$ and $a'=2$.
Since the surfaces  $S_{a,0}$ and $S_{a',0}$ are isometric,
and the boundaries $\partial N_{a,0}$ and $\partial N_{a',0}$ 
are union of one or two copies of these surfaces, 
we can glue one or two copies of the $3$-manifolds $N_{a,0}$  and $N_{a',0}$
along their boundaries 
and get a connected finite volume hyperbolic manifold $M_0$ with no boundary. 
We write 
$$
M_0=\Gamma_0\ba \m H^3.
$$ 
By \cite{GP}, the lattice $\Gamma_0$ of $G$  is  non-arithmetic. 
The group $\Gamma_a$ is also naturally a subgroup of $\Gamma_{0}$
and there is a natural covering map
$$
p:M_a\to M_0.
$$

Theorem \ref{nono} follows from the following:
\begin{prop}  
\label{na} Let $S$ denote the boundary of the convex core of $M_a$.
Let  $P\subset M_a$ be a geodesic plane that intersects  $S$ but is neither contained in $S$ nor orthogonal to $S$.
Then  the image $p(P)$  is dense in $M_0$. 

In particular, there exists a closed geodesic plane $P\subset M_a$
that intersects $M_a^*$ and whose image $p(P)$ is dense in $M_0$.
\end{prop}

We first need to compute the angle $\theta_{a,v}$ between two rational planes in $\m H^3$; the following lemma follows from a direct computation.

\begin{lemma}
\label{angle1} Let $v=(w,x_4)$ with $w\in \m R^3$, $x_4\in \m R$
with $q_a(v)>0$.

The intersection of the two geodesic planes  $\m H_v$ and $\m H_{v_0}$ is a geodesic line $\m D_w$, independent of
 $a$ and $x_4$.

The angle $\theta_{a,v}:=\angle(\m H_v,\m H_{v_0})$ between these
geodesic planes  is given by 
\be
\label{cos}
\cos^2(\theta_{a,v})=
\frac{\langle v_0, v\rangle_{q_a}^2}{q_a(v_0) q_a(v)} =
\frac{ax_4^2}{q_0(w)+ax_4^2}.
\ee
\end{lemma}

Therefore we have:
\begin{cor}
\label{angle2}  
Fix $w\in \m Q^3-\{0\}$ and set 
$$
\Theta_{a,w}:=\{ \theta_{a,(w, x)}:  x\in  \m Q-\{ 0\} \}.
$$
If ${a}/{a'}$ is not a square in $\m Q$, then
$$\Theta_{a,w}\cap \Theta_{a',w}=\emptyset .$$
\end{cor}

\begin{proof}
For $v=(w,x_4)$ and  $v'=(w,x'_4)$ with $x_4$ and $x'_4$ in $\m Q-\{ 0\}$, by Formula \eqref{cos},
an equality $\theta_{a,v}=\theta_{a',v'}$
would imply $a x_4^2=a'{x'_4}^2$, and hence $a/a'\in \q^2$.  This proves the claim. \end{proof}

\begin{proof}[Proof of Proposition \ref{na}]
By Theorem \ref{m2}, which is due to Shah and Ratner independently (\cite{Sh}, \cite{Rn})
for a hyperbolic manifold of finite volume, 
the geodesic plane $Q:=p(P)$ is either closed or dense in $M_0$. 
We assume by contradiction that $Q$ is  closed in $M_0$. 

Since $Q\cap N_{a,0}$ is closed, as explained in 
Theorem \ref{arith1},   $Q$ must be the image in $M_0$ 
of a rational plane $\m H_v$ with $v=(w,x_4)\in \m Q^4$,
in the model \eqref{hqa} of $\m H^3$ given by the quadratic form $q_a$.
Similarly, since $Q\cap N_{a',0}$ is closed,
$Q$ must be the image in $M_0$ 
of a rational plane $\m H_{v'}$ with $v'=(w',x'_4)\in \m Q^4$,
in the model of $\m H^3$ given by the quadratic form $q_{a'}$.

We can choose these lifts such that
the two intersection geodesics $\m D_w= \m H_{v}\cap \m H_{v_0}$ and $\m D_{w'}= \m H_{v'}\cap \m H_{v_0}$ are equal 
and the corresponding two angles between these intersecting planes are equal.
This says that 
$w$ and $w'$ are equal up to a multiplicative factor,
and that  $\theta_{a,v}=\theta_{a',v'}$. This contradicts 
Corollary \ref{angle2}.
\end{proof}

After seeing the result proven in the paper of Fisher, Lafont, Miller and Stover \cite{FLMS}, we realized that Proposition \ref{na}, 
Theorem \ref{mainacy}(2) and the main result in \cite{Sh} together can also be used to show that the non-arithmetic manifolds in section 12.5 can have at most finitely many properly immersed geodesic planes. For a more general result in this direction, see \cite{FLMS}.  

\end{document}